\documentclass[11pt,a4paper,reqno]{amsart}
\usepackage{amssymb,amsmath,amsfonts,amscd}
\usepackage{mathrsfs}
\usepackage{latexsym}
\usepackage{amssymb}
\usepackage{amsthm}
\usepackage{epsfig}
\usepackage[breaklinks=true]{hyperref} 
\usepackage{fontenc}
\usepackage{pifont}
\usepackage{wasysym}
\usepackage{graphicx}
\usepackage{hyperref}
\usepackage{pifont}
\usepackage{epstopdf}
\usepackage{booktabs}
\usepackage{subfig,float}

\usepackage{multirow} 
\allowdisplaybreaks

\numberwithin{equation}{section}
\newtheorem{theorem}{Theorem}[section]
\newtheorem{lemma}{Lemma}[section]

\newtheorem{remark}{Remark}[section]

\newcommand{\uo}{u_{1}}
\newcommand{\ut}{u_{2}}
\newcommand{\vo}{v_{1}}
\newcommand{\vt}{v_{2}}

\newcommand{\sio}{\mu_1}
\newcommand{\sit}{\mu_2}
\newcommand{\alo}{a_1}
\newcommand{\alt}{a_2}
\newcommand{\beo}{b_1}
\newcommand{\bet}{b_2}

\newcommand{\uost}{u_{1}^*}
\newcommand{\vost}{v_{1}^*}
\newcommand{\utst}{u_{2}^*}
\newcommand{\vtst}{v_{2}^*}

\newcommand{\uoe}{u_{1}^e}
\newcommand{\voe}{v_{1}^e}
\newcommand{\ute}{u_{2}^e}
\newcommand{\vte}{v_{2}^e}

\newcommand{\uot}{\uo(\cdot,t)}
\newcommand{\utt}{\ut(\cdot,t)}
\newcommand{\vot}{\vo(\cdot,t)}
\newcommand{\vtt}{\vt(\cdot,t)}

\newcommand{\uos}{\uo(\cdot,s)}
\newcommand{\uts}{\ut(\cdot,s)}
\newcommand{\vos}{\vo(\cdot,s)}
\newcommand{\vts}{\vt(\cdot,s)}

\newcommand{\gvo}{{\nabla \vo}}
\newcommand{\gvt}{{\nabla \vt}}
\newcommand{\lvo}{{\Delta \vo}}

\newcommand{\guo}{{\nabla \uo}}
\newcommand{\gut}{{\nabla \ut}}

\newcommand{\rsn}{\mathbb{R}^n}
\newcommand{\lis}{\mathcal{L}^{\infty}(\Omega)}
\newcommand{\los}{\mathcal{L}^{1}(\Omega)}
\newcommand{\lts}{\mathcal{L}^{2}(\Omega)}
\newcommand{\ltps}{\mathcal{L}^\frac{2}{q}(\Omega)}
\newcommand{\lps}{\mathcal{L}^{q}(\Omega)}

\newcommand{\lqs}{\mathcal{L}^{p}(\Omega)}
\newcommand{\wsq}{\mathcal{W}^{1,p}(\Omega)}
\newcommand{\wsin}{\mathcal{W}^{1,\infty}(\Omega)}
\newcommand{\cso}{ \mathcal{C}^{0}(\overline{\Omega})}
\newcommand{\cts}{\mathcal{C}^{2}(\overline{\Omega})}

\newcommand{\momeg}{|\Omega|}

\newcommand{\nuo}{\big\|\uo\big\|}
\newcommand{\nut}{\big\|\ut\big\|}
\newcommand{\nvo}{\big\|\vo\big\|}
\newcommand{\nvt}{\big\|\vt\big\|}

\newcommand{\nuot}{\big\|\uot\big\|}
\newcommand{\nutt}{\big\|\utt\big\|}
\newcommand{\nvot}{\big\|\vot\big\|}
\newcommand{\nvtt}{\big\|\vtt\big\|}

\newcommand{\mguo}{|\guo|}

\newcommand{\uoi}{u_{10}}
\newcommand{\uti}{u_{20}}
\newcommand{\voi}{v_{10}}
\newcommand{\vti}{v_{20}}

\newcommand{\tmax}{T_{\mathrm{max}}}

\newcommand{\intts}{\int^t_{s_0}}

\newcommand{\intT}{\int_{s_0}^T}

\newcommand{\ints}{\int_{\Omega}}

\newcommand{\uop}{u^q_1}
\newcommand{\utp}{u^q_2}
\newcommand{\uopo}{u^{q-1}_1}
\newcommand{\uopp}{u^{q+1}_1}
\newcommand{\uopt}{u^{q-2}_1}

\newcommand{\fp}{\frac{1}{q}}

\newcommand{\fpop}{\frac{q+1}{q}}
\newcommand{\dt}{\frac{\mathrm{d}}{\mathrm{d}t}}

\newcommand{\mlv}{|\lvo|}
\newcommand{\mlvop}{|\lvo|^{q+1}}

\newcommand{\utpp}{u^{q+1}_2}

\newcommand{\epts}{e^{-(q+1)(t-s)}}

\usepackage{pdfpages}
\graphicspath{{PPCDS_Figures/},{PPCDS_Pattern_Figures/}}

\usepackage{array}
\begin{document}

\title[Global dynamics and patterns in a predator-prey system with two chemicals]{Global dynamics and Diffusion-Driven Pattern Formation in a Predator-Prey System with Two Chemicals}

\author[S. Gnanasekaran]{Gnanasekaran Shanmugasundaram}
\address[GS]{Department of Mathematics, National Institute of Technology Tiruchirappalli, Tamilnadu 620015, India}
\curraddr{}
\email{sekaran@nitt.edu}
\thanks{}

\author[J. Saha]{Jitraj Saha$^*$}
\address[JS]{Department of Mathematics, National Institute of Technology Tiruchirappalli, Tamilnadu 620015, India}
\curraddr{}
\email{jitraj@nitt.edu}
\thanks{$^*$Corresponding author}

\author[O. D. Makinde]{Oluwole Daniel Makinde}
\address[ODM]{Faculty of Military Science, Stellenbosch University, Stellenbosch, South Africa}
\curraddr{}
\email{makinded@gmail.com}
\thanks{}

\author[J. Chattopadhyay]{Joydev Chattopadhyay}
\address[JC]{Agricultural \& Ecological Research Unit, Indian Statistical Institute Kolkata, West Bengal 700108, India}
\curraddr{}
\email{joydev@isical.ac.in}
\thanks{}

\keywords{Cross-diffusion systems, Global existence, Asymptotic stability, Pattern formation, Numerical Simulations}

\subjclass[2020]{35A01, 35A09, 35B36, 35B40, 37N25, 65N06}

\begin{abstract}
	This work analyzes a predator–prey cross-diffusion system coupled with two chemical substances under homogeneous Neumann boundary conditions in a bounded domain $\Omega \subset \mathbb{R}^n$ $(n \geq 2)$ with smooth boundary $\partial \Omega$. Under appropriate conditions on the model parameters, the global existence of classical solutions is established. Furthermore, by constructing a suitable Lyapunov functional, the asymptotic stability of the spatially homogeneous steady state is proved. The emergence of spatial patterns induced by diffusion-driven instability is also investigated. Owing to the complexity of the resulting four-equation system, the criteria for Turing bifurcation are derived numerically rather than analytically. Numerical simulations are performed to generate Turing bifurcation diagrams, illustrating the dynamical responses of the system to variations in the predation rate. These results provide new insights into the role of predation intensity in the formation of spatial patterns in predator–prey systems mediated by two chemical substances.
\end{abstract}

\maketitle
\section{Introduction and motivation} 
Predator–prey dynamics have long stood at the core of mathematical ecology, beginning with the pioneering works of Lotka \cite{lotka1925} and Volterra \cite{volterra1927}, who established a theoretical framework for population interactions through coupled nonlinear differential equations. These nonlinear interactions in the predator-prey systems can give rise to rich spatiotemporal behaviors that demand rigorous mathematical analysis. One of the most remarkable manifestations of such dynamics is the emergence of Turing patterns \cite{turing1952}, which arise in reaction–diffusion systems due to interactions between activator and inhibitor components with different diffusion rates. These mechanisms provide explanations for a wide range of biological patterns, including animal coat markings and tissue morphogenesis. Segel and Jackson \cite{segel1972} were among the first to recognize that Turing’s theory could be extended to ecological systems, thereby linking pattern formation theory with population dynamics.

Traditional models of ecological pattern formation often rely on self-diffusion and assume time-independent parameters. However, self-diffusion alone is not sufficient to generate or sustain spatial patterns \cite{Chattopadhyay1994, McLaughlin1992}. This limitation has motivated the inclusion of the concept of cross-diffusion, which describes the movement of one species being dependent on the concentration gradient of another species. Cross-diffusion has been widely studied \cite{NAhmed2025, LNGuin2017, Gurtin1974, GHu2015, Okubo1980, GQSun2012, VTiwari2020} and is now recognized as a key mechanism driving spatial heterogeneity and pattern formation in ecological systems, including Lotka–Volterra-type models \cite{joydev1996}. 
Recent advances in reaction–diffusion predator–prey models have highlighted the crucial roles of spatial heterogeneity, behavioral responses, and nonlinear dispersal mechanisms in shaping ecosystem stability and pattern formation. A broad spectrum of model formulations including cross-diffusion, Allee effects, fear effects, harvesting, and time delays has been extensively investigated. These studies have established rigorous criteria for local and global stability, as well as for Hopf and Turing bifurcations, which give rise to complex spatiotemporal dynamics. In this context, we briefly review some recent contributions reported in the literature.

A substantial body of work has examined pattern formation and stability in reaction–diffusion predator–prey systems under increasingly realistic mechanisms. Ranjit et al. \cite{RKUpadhyay2014} analyzed a model with self- and cross-diffusion, establishing conditions for local/global stability and diffusion-driven instability. Their results showed that cross-diffusion fundamentally alters classical Turing thresholds, enabling pattern formation in regimes inaccessible to standard diffusion. Santu and Swarup \cite{SGhorai20161} investigated Hopf and Turing bifurcations in the presence of supplementary food, demonstrating that additional resources can shift stability boundaries and suppress spatiotemporal chaos. Similarly, Qian and Jianhua \cite{QCao2021} incorporated chemotaxis, showing that directed movement enlarges the admissible parameter space for instability and induces patterns absent in purely diffusive systems.

Subsequent studies extended these frameworks to more complex ecological interactions. Tamko et al. \cite{BTMbopda2021} explored a hepatitis B virus model within a predator–prey framework, revealing complex spatial structures driven by competitive and commensal interactions. Santu et al. \cite{SGhorai2023} further examined cross-diffusion and supplementary food, confirming that resource enrichment significantly enhances pattern diversity. Yong et al. \cite{YWang2023} analyzed a harvested system with Michaelis–Menten response, establishing Turing conditions and observing both homogeneous and heterogeneous periodic dynamics. In a tri-trophic context, Bhaskar et al. \cite{BChakraborty2024} identified multiple bifurcation scenarios—including Hopf, wave, and mixed instabilities—and showed that diffusion can suppress oscillations and stabilize coexistence. Gourav et al. \cite{GMandal2024} incorporated fear and anti-predator behavior, demonstrating that behavioral effects alone can generate both Turing and non-Turing spatial structures.

Recent contributions emphasize structural complexity and environmental heterogeneity. Shunjie et al. \cite{SLi2025} highlighted the role of symmetry and domain geometry in pattern selection, showing that spatial configuration critically shapes emergent structures. Lakpa et al. \cite{LTBhutia2025} extended the Rosenzweig–MacArthur model with harvesting and cross-diffusion, revealing transitions between stable coexistence, oscillations, and chaotic attractors. Esita et al. \cite{EDas2025} demonstrated that Allee effects and harvesting can destabilize spatial systems and generate rich transient dynamics. Gourav et al. \cite{GMandal2025} showed that hunting cooperation and Allee effects, combined with cross-diffusion, significantly amplify instability and promote diverse pattern formation. Pallav et al. \cite{PJPal2025} further emphasized the role of strong Allee effects and habitat fragmentation in driving extinction and complex spatial dynamics. Muhammad et al. \cite{MWYasin2025} provided a systematic classification of dynamical regimes using a positivity-preserving numerical scheme, ensuring both analytical and computational consistency.

Very recent work has focused on eco-epidemic and control mechanisms. Suvankar et al. \cite{SMajee2026} studied an eco-epidemic predator–prey model with refuge and competition, establishing conditions for Turing instability and chaotic dynamics. Their results showed that uncontrolled chaos can destabilize ecosystems, while time-delay feedback control effectively restores stability.

To the best of our knowledge, most existing studies on pattern formation in predator–prey systems focus on either self-diffusion or direct cross-diffusion between species. In contrast, cross-diffusion mediated by self-produced chemical signals has received comparatively limited attention, despite its strong biological relevance. Indeed, classical and modern predator-prey frameworks show that species interactions are often directly mediated through diffusible chemical cues, which can significantly influence aggregation, dispersal, and pattern formation dynamics. Such signal-mediated movement introduces an additional layer of nonlinear coupling that cannot be captured by standard cross-diffusion terms alone. Therefore, incorporating chemically mediated cross-diffusion is essential for developing more realistic models and for understanding how indirect interactions reshape Turing instability and spatial self-organization in ecological systems.

\section{Mathematical model}
Chemically mediated movement plays a fundamental role in ecological interactions, as many organisms release diffusible cues that influence the movement and spatial organization of other species \cite{murray2002, murray2003}. Such mechanisms naturally induce cross-diffusion and can trigger Turing-type instabilities, leading to complex spatial patterns.  A biologically relevant example of such mechanisms is found in bacteria–bacteriophage systems \cite{STAbedon2009, JSWeitz2013}. Bacterial species, including Escherichia coli and Pseudomonas, release autoinducers and metabolic by-products that can enhance bacteriophage replication and facilitate their spatial spread. Conversely, bacteriophages release inhibitory substances during infection and lysis, suppressing bacterial growth and dispersal. Similar chemically mediated interactions are observed in algae–zooplankton systems \cite{DCOThornton2014, LJiang2005}, where algae produce dissolved organic carbon that promotes zooplankton growth, while zooplankton release kairomones that inhibit algal reproduction and movement.

Motivated by these observations, we consider a predator–prey system in which movement is governed by bidirectional chemotaxis: prey release a signal $v_1$ that attracts predators, while predators emit a signal $v_2$ that repels prey. These mechanisms induce advective fluxes along chemical gradients, giving rise to cross-diffusion as an emergent effect. The resulting system is therefore given by
\begin{align}
	\left\{
	\begin{array}{lll}
		&u_{1t}=\nabla \cdot (d_{11}\nabla\uo+d_{12}\nabla \vo)+\uo(\sio-\lambda_1\uo+\eta_1\ut),\hspace{0.5cm} &x\in\Omega,\, t>0,\\
		&u_{2t}=\nabla \cdot (d_{21}\nabla\ut-d_{22}\nabla \vt)+\ut(\sit-\lambda_2\ut-\eta_2\uo), &x\in\Omega,\, t>0,\\
		&v_{1t}=d_3\Delta\vo+\alo\ut-\beo\vo, &x\in\Omega,\, t>0,\\
		&v_{2t}=d_4\Delta \vt+\alt \uo-\bet\vt, &x\in\Omega,\, t>0,\\
		&\frac{\partial\uo}{\partial\nu}=\frac{\partial\ut}{\partial\nu}=\frac{\partial\vo}{\partial\nu}=\frac{\partial\vt}{\partial\nu}=0, &x\in\partial\Omega,\, t>0,\\
		&\uo(x,0)=u_{10}, \ut(x,0)=u_{20}, \vo(x,0)=v_{10}, \vt(x,0)=v_{20}, &x\in\Omega,
	\end{array}
	\right.\label{1.1}
\end{align}
in an open, bounded domain $\Omega\subset \mathbb{R}^n$ with smooth boundary $\partial\Omega$. Here, $\nu$ represents the unit outward normal on $\partial\Omega$. The unknown functions $\uo=\uo(x,t)$ and $\ut=\ut(x,t)$ describes the density of the population of predator and prey, respectively, and $\vo=\vo(x,t)$, $\vt=\vt(x,t)$ denotes the concentration of chemical attractants produced by prey and predator respectively. Here, the parameters $d_{11}, d_{12}$, $d_{21}, d_{22}$, $d_3$, $d_4$, $\sio$, $\sit$, $\lambda_1, \lambda_2$, $\eta_1$, $\eta_2$, $\alo, \alt$, $b_1$ and $\bet$ are positive constants and the initial data $u_{10}, u_{20}$, $v_{10}$ and $v_{20}$ are non-negative functions. The constants $d_{11}, d_{21}$, $d_3$ and $d_4$ are labeled as self diffusion coefficients, where as $d_{12}$ and $d_{22}$ are the cross diffusion coefficients, $\nabla\cdot (d_{12}\nabla v_1)$ denotes the directional predator movement towards the substance produced by the prey and the term $-\nabla\cdot (d_{22}\nabla v_2)$ describes the directional prey movement away from the substance produced by the predator. The growth rates of predator and prey are labeled as $\sio$ and $\sit$ respectively. The constants $\lambda_1, \lambda_2$ are known as the interaction between the species themselves and $\eta_1$, $\eta_2$ represent the interaction between other species. The parameters $\alo$ and $\alt$ characterizes the production rates of the signals by the prey and predator, $\beo$ and $\bet$ characterizes the decay rate of the chemical attractants. We assume the following conditions to ensure the parabolicity of \eqref{1.1}
\begin{align}
	4d_{11}d_3>d_{12}^2 \quad \text{and} \quad 4d_{21}d_4>d_{22}^2.\label{1.3}
\end{align}
Further, assume that the initial data $\uoi$, $\uti$, $\voi$ and $\vti$ satisfy
\begin{align}
	\left\{
	\begin{array}{llll}
		&\uoi, \uti\in\cso,\quad \mbox{with} \quad \uoi, \uti \geq 0\quad\mbox{in}\,\: \Omega,\\
		&\voi, \vti\in\wsin, \quad \mbox{with} \quad \voi, \vti\geq 0\quad\mbox{in}\,\: \Omega.
	\end{array}
	\right.\label{1.2}
\end{align}
Motivated by the works discussed in Section 1, this paper establishes the global existence and boundedness of classical solutions to \eqref{1.1}, proves their global asymptotic stability, and demonstrates the emergence of Turing patterns induced by the system.

The remainder of the paper is organized as follows. Section 3 presents preliminary results and proves local existence of classical solutions. Section 4 establishes global existence and boundedness. Stability analysis is carried out in Section 5. Section 6 investigates pattern formation driven by cross-diffusion–induced instability. Finally, Section 7 concludes with a summary of the main findings.

The principal theorems for the proposed system are stated below.
\begin{theorem}\label{t1}
	Let $\Omega \subset \mathbb{R}^n$ ($n \ge 2$) be a bounded domain with smooth boundary, and assume that \eqref{1.3} is satisfied. Then, for any nonnegative initial data $(u_{10}, u_{20}, v_{10}, v_{20})$ fulfilling \eqref{1.2}, the system \eqref{1.1} admits a unique classical solution $(u_1, u_2, v_1, v_2)$, which remains uniformly bounded, in the sense that for all $t>0$
	\begin{align*}
		\nuot_{\lis}+\nutt_{\lis}+\nvot_{\wsin}+\nvtt_{\wsin}\leq C,
	\end{align*}
	where  $C>0$ is a constant.
\end{theorem}

Let $(u_1, u_2, v_1, v_2)$ be the classical solution of \eqref{1.1}  and $(\uoe, \ute, \voe, \vte)$ be the equilibria of \eqref{1.1}, satisfying the following system
\begin{align*}
	\left\{
	\begin{array}{rl}
		\uoe(\sio-\lambda_1\uoe+\eta_1\ute)=&0,\\
		\ute(\sit-\lambda_2\ute-\eta_2\uoe)=&0,\\
		a_1\ute-b_1\voe=&0,\\
		a_2\uoe-b_2\vte=&0.
	\end{array}
	\right.
\end{align*}
The system has four equilibria
\begin{align*}
	(0, 0, 0, 0), \quad
	\displaystyle{\left(0, \frac{\sit}{a_2}, \frac{\alo\sit}{a_2\beo}, 0\right)},\quad
	\displaystyle{\left(\frac{\sio}{a_1}, 0, 0, \frac{\alt\sio}{a_1\bet}\right)},\quad
	(\uost, \utst, \vost, \vtst).
\end{align*}
The coexistence equilibrium point $(\uost, \utst, \vost, \vtst)$ is given by
\begin{align*}
	\uost=\frac{a_2\sio+\eta_1\sit}{a_1a_2+\eta_1\eta_2},\:\:
	\utst=\frac{a_1\sit-\eta_2\sio}{a_1a_2+\eta_1\eta_2}, \:\:
	\vost=\frac{a_1(a_1\sit-\eta_2\sio)}{b_1(a_1a_2+\eta_1\eta_2)},\:\:
	\vtst=\frac{a_2(a_2\sio+\eta_1\sit)}{b_2(a_1a_2+\eta_1\eta_2)}
\end{align*}
provided $a_1\sit>\eta_2\sio$.

\begin{theorem}\label{t2}
	Suppose that the assumptions of Theorem \ref{t1} hold true and let $\eta_2<\frac{\sit a_1}{\sio}$. If the parameters satisfy the relations
	\begin{align*}
		d_{12}^2<&\min\left\{\frac{16d_{11} d_3 b_1\eta_1a_2(a_1a_2+\eta_1\eta_2)\nuo^2_{\lis}}{a_1^2\eta_2(a_2\sio+\eta_1\sit)}, 4d_{11}d_3\right\},\\
		d_{22}^2<&\min\left\{\frac{16d_{21}d_4b_2\eta_2a_1(a_1a_2+\eta_1\eta_2)\nut^2_{\lis}}{a_2^2\eta_1(a_1\sit-\eta_2\sio)}, 4d_{21}d_4\right\},
	\end{align*}
	then the nonnegative classical solution $(u_1, u_2, v_1, v_2)$ of the system \eqref{1.1} exponentially converges to the unique positive equilibrium point $(\uost, \utst, \vost, \vtst)$ uniformly in $\Omega$ as $t\to \infty$.
\end{theorem}

\begin{theorem}\label{t3}
	Suppose that the assumptions of Theorem \ref{t1} hold true and let $\eta_2\geq\frac{\sit a_1}{\sio}$. If the parameters satisfy the relation
	\begin{align*}
		d_{12}^2<\min\left\{\frac{16d_{11} d_3b_1\eta_1a_1a_2\nuo_{\lis}^2}{a_1^2\sio \eta_2}, 4d_{11}d_3 \right\},
	\end{align*}
	then the nonnegative classical solution $(u_1, u_2, v_1, v_2)$ of the system \eqref{1.1} converges to the semi-trivial equilibrium point $\displaystyle{\left(\frac{\sio}{a_1}, 0, 0, \frac{\sio a_2}{a_1b_2}\right)}$ uniformly in $\Omega$ as $t\to \infty$.
\end{theorem}

\section{Preliminaries and local existence of solutions}

\begin{lemma}[Local Existence]\label{l1}
	Let $\Omega \subset \mathbb{R}^n$ ($n \ge 2$) be a bounded domain with smooth boundary. Suppose the initial data $\uoi$, $\uti$, $\voi$, and $\vti$ satisfy \eqref{1.2} for some $p>n$, and assume that condition \eqref{1.3} holds. Then there exists a maximal time $\tmax \in (0,\infty]$ for which the system \eqref{1.1} admits a unique solution $(u_1, u_2, v_1, v_2)$ satisfying
	\begin{align*}
		\uo, \ut&\in { \mathcal{C}^{0}}\left(\overline{\Omega}\times\left.\left[0,\tmax\right.\right)\right)\cap { \mathcal{C}^{2,1}}\left(\overline{\Omega}\times\left(0,\tmax\right)\right),\nonumber\\
		\vo, \vt&\in { \mathcal{C}^{0}}\left(\overline{\Omega}\times\left.\left[0,\tmax\right.\right)\right)\cap {\mathcal{C}^{2,1}}\left(\overline{\Omega}\times\left(0,\tmax\right)\right)\cap {\mathcal{L}^{\infty}_{loc}}\left(\left.\left[0,\tmax\right.\right);\wsq\right).\nonumber
	\end{align*}
	Furthermore, either $\tmax=\infty$ or
	\begin{align}
		\lim_{t\to \tmax}\left(\nuo_{\lis}+\nut_{\lis}+\nvo_{\wsq}+\nvt_{\wsq}\right)= \infty.\label{l2.2.1}
	\end{align}
\end{lemma}
\begin{proof}
	The standared arguments involving the quasi-linear parabolic theory is used to prove the lemma. 
	Let $w=(u_1, u_2, v_1, v_2)^{T}\in\mathbb{R}^4$, then system \eqref{1.1} can be reformed as
	\begin{align*}
		\left\{
		\begin{array}{rll}
			&w_t=\nabla \cdot (\mathcal{A}\nabla w)+\mathcal{F}(w),\hspace{0.5cm} &x\in\Omega,\, t>0,\\
			&\frac{\partial w}{\partial\nu}= 0, &x\in\partial\Omega,\, t>0,\\
			& w(x,0)=w_0, &x\in\Omega,
		\end{array}
		\right.
	\end{align*}
	where
	\begin{align*}
		\mathcal{A}=
		\begin{pmatrix}
			d_{11}& 0 & d_{12} & 0\\
			0 & d_{21} & 0 & -d_{22} \\
			0 & 0 & d_3 & 0\\
			0 & 0 & 0 & d_4
		\end{pmatrix},
		\hspace*{1cm}
		\mathcal{F}(w)=
		\begin{pmatrix}
			u_1(\mu_1-\lambda_1u_1+\eta_1u_2)\\
			u_2(\mu_2-\lambda_2u_2-\eta_2u_1)\\
			a_1 u_2 -b_1 v_1\\
			a_2 u_1 - b_2 v_2
		\end{pmatrix}.
	\end{align*}
	Under the parabolicity condition stated in \eqref{1.3}, Amann’s theory (Theorem 14.4 in \cite{amann1993}) ensures the existence of a weak maximal solution. Moreover, by Theorem 14.6 in \cite{amann1993}, this solution is in fact classical and satisfies \eqref{1.1} pointwise.
	
	\emph{Nonnegativity of the solution: }Define the negative parts
	$u_i^-:=\max\{-u_i,0\}$ and $v_i^-:=\max\{-v_i,0\}$ for $i=1,2$, and set	$\mathcal{N}(t):=\sum_{i=1}^2 \|u_i^-(\cdot,t)\|_{\lts}^2
	+ \sum_{i=1}^2 \|v_i^-(\cdot,t)\|_{\lts}^2$. We multiply each equation of \eqref{1.1} by the negative part of the corresponding
	component and integrate over $\Omega$. Using the identity
	$\int_\Omega (\nabla\!\cdot F)\, w^-
	= -\int_\Omega F\cdot \nabla w^-$
	and the fact that $\nabla w = -\nabla w^-$ on $w<0$, we obtain the diffusion contributions
	\begin{align*}
		\int_\Omega \nabla\!\cdot(d_{11}\nabla u_1 + d_{12}\nabla v_1)\,u_1^-
		&= d_{11}\int_\Omega |\nabla u_1^-|^2\,
		- d_{12}\int_\Omega \nabla v_1\cdot\nabla u_1^-,
		\\
		\int_\Omega \nabla\!\cdot(d_{21}\nabla u_2 - d_{22}\nabla v_2)\,u_2^-
		&= d_{21}\int_\Omega |\nabla u_2^-|^2\,
		+ d_{22}\int_\Omega \nabla v_2\cdot\nabla u_2^-, 
		\\
		\int_\Omega d_3\Delta v_1\,v_1^-\,
		&= d_3 \int_\Omega |\nabla v_1^-|^2, 
		\\
		\int_\Omega d_4\Delta v_2\,v_2^-\,
		&= d_4 \int_\Omega |\nabla v_2^-|^2. 
	\end{align*}
	Therefore, we get
	\begin{align}
		\frac12\frac{d}{dt}\mathcal{N}(t)
		&+ d_{11}\|\nabla u_1^-\|_{\lts}
		+ d_{21}\|\nabla u_2^-\|_{\lts}^2
		+ d_3\|\nabla v_1^-\|_{\lts}^2
		+ d_4\|\nabla v_2^-\|_{\lts}^2 \notag\\
		&= d_{12}\int_\Omega \nabla v_1\cdot\nabla u_1^-
		- d_{22}\int_\Omega \nabla v_2\cdot\nabla u_2^-
		- \sum_{i=1}^4\int_\Omega \mathcal{F}_i(w_i)\,w_i^-,\label{l.2.2.2}
	\end{align}
	where $w_1^- = u_1^-$, $w_2^- = u_2^-$, $w_3^- = v_1^-$ and $w_4^- = v_2^-$.
	Applying Young's inequality, we get
	\begin{align*}
		\Big|d_{12} \nabla v_1 \cdot \nabla u_1^-\Big|
		&\le \frac{d_{12}^2}{2d_3} \big|\nabla u_1^-\big|^2
		+ \frac{d_3}{2} \big|\nabla v_1^-\big|^2, 
		\\
		\Big|d_{22} \nabla v_2 \cdot \nabla u_2^-\Big|
		&\le \frac{d_{22}^2}{2d_4} \big|\nabla u_2^-\big|^2
		+ \frac{d_4}{2} \big|\nabla v_2^-\big|^2. 
	\end{align*}
	Using the structural assumptions
	$4 d_{11} d_3 > d_{12}^2$ and $4 d_{21} d_4 > d_{22}^2$,
	we conclude that the coefficients
	$d_{11}-\frac{d_{12}^2}{2d_3} > 0$ and
	$d_{21}-\frac{d_{22}^2}{2d_4} > 0$,
	and hence there exists $c_1>0$ such that
	\begin{align}
		&d_{11}\|\nabla u_1^-\|_{\lts}^2
		+ d_3\|\nabla v_1^-\|_{\lts}^2
		- d_{12}\!\int_\Omega \nabla v_1\!\cdot\!\nabla u_1^- \ge
		c_1\Big(\|\nabla u_1^-\|_{\lts}^2 + \|\nabla v_1^-\|_{\lts}^2\Big), \label{l.2.2.6}
		\\
		&d_{21}\|\nabla u_2^-\|_{\lts}^2
		+ d_4\|\nabla v_2^-\|_{\lts}^2
		+ d_{22}\!\int_\Omega \nabla v_2\!\cdot\!\nabla u_2^- \ge
		c_1\Big(\|\nabla u_2^-\|_{\lts}^2 + \|\nabla v_2^-\|_{\lts}^2\Big). \label{l.2.2.7}
	\end{align}
	Since the reaction vector $\mathcal{F}_i(w_i)$ is quasi-positive and $w=-w^-$ on $w<0$, we get
	\begin{align*}
		-\mathcal{F}_1(u_1)u_1^-=-\uo(\sio-\lambda_1u_1+\eta_1u_2)u_1^-
		=(u_1^-)^2(\sio-\lambda_1u_1+\eta_1u_2)
		\leq (\sio+\lambda_1C_1+\eta_1C_2)(u_1^-)^2
	\end{align*}
	here we using $\nuo, \nut\leq C$ from local existence. Similarly,
	\begin{align*}
\begin{array}{llll}
		-\mathcal{F}_2(u_2)u_2^-=\sit(u_2^-)^2-\lambda_2u_2(u_2^-)^2-\eta_2u_1(u_2^-)^2
		&\leq (\sit+\lambda_2C_3+\eta_2C_4)(u_2^-)^2.\\
		-\mathcal{F}_3(v_1)v_1^-=-a_1u_2v_1^-+b_1v_1v_1^-&=a_1u_2^-v_1^--b_1v_1^-v_1^-\leq a_1u_2^-v_1^-,\\
		-\mathcal{F}_4(v_2)v_2^-=-a_2u_1v_2^-+b_2v_2v_2^-&=a_2u_1^-v_2^--b_2v_2^-v_2^-\leq a_2u_1^-v_2^-.
	\end{array}
	\end{align*}
	Combining all the terms and then applying Young's inequality to get
	\begin{align}
		-\sum_{i=1}^4 \int_\Omega \mathcal{F}_i(w_i)\,w_i^-
		\le& \ints (\sio+\lambda_1C_1+\eta_1C_2)(u_1^-)^2+\frac{a_2}{2}(u_1^-)^2+\ints(\sit+\lambda_2C_3+\eta_2C_4)(u_2^-)^2\nonumber\\
		&+\frac{a_1}{2}(u_2^-)^2+\ints\frac{a_1}{2}(v_1^-)^2+\frac{a_2}{2}(v_2^-)^2\label{l.2.2.8}
	\end{align}
	Substitute \eqref{l.2.2.6}-\eqref{l.2.2.8} in to \eqref{l.2.2.2} yields
	\begin{align*}
		\dt \mathcal{N}(t)
		+ 2c_1 \sum_{i=1}^4 \|\nabla w_i^-\|_{\lts}^2
		\le 2c_2\,\mathcal{N}(t). 
	\end{align*}
	Dropping the nonnegative gradient term and applying Gr\"onwall's inequality,	$\mathcal{N}(t) \le \mathcal{N}(0)e^{2c_2t}$.
	Because the initial data are nonnegative, $\mathcal{N}(0)=0$, hence
	$\mathcal{N}(t)=0$ for all $t\ge0$. Therefore $u_i^-=v_i^-\equiv0$ and	$u_1,\,u_2,\,v_1,\,v_2 \ge 0$ in $\Omega\times[0,\infty)$.
\end{proof}

Let $s_0 \in (0,\tmax)$ be such that $s_0 < 1$. By Lemma \ref{l1}, we have $\uo(\cdot,s_0), \ut(\cdot,s_0), \vo(\cdot,s_0),$ $\vt(\cdot,s_0)\in\cts$ with $\frac{\partial {v}_i(\cdot, s_0)}{\partial \nu}=0, i=1,2$. Choose a constant $C>0$ for which
\begin{align}
	&\sup\limits_{0\leq s\leq s_0}\big\|\uos\big\|_{\lis}\leq C, \qquad \sup\limits_{0\leq s\leq s_0}\big\|\uts\big\|_{\lis}\leq C, \nonumber\\ &\sup\limits_{0\leq s\leq s_0}\big\|\vos\big\|_{\lis}\leq C, \qquad \big\|\Delta v_1(\cdot, s_0)\big\|_{\lis}\leq C. \label{l3.0}\\
	&\sup\limits_{0\leq s\leq s_0}\big\|\vts\big\|_{\lis}\leq C, \qquad \big\|\Delta v_2(\cdot, s_0)\big\|_{\lis}\leq C.\nonumber
\end{align}

\begin{lemma}[\cite{zheng2015}]\label{l.2.1}
	Let $y$ be a positive absolutely continuous function on $(0, \infty)$ that satisfies
	\begin{align*}
		\left\{\hspace*{-0.5cm}
		\begin{array}{rl}
			&y'(t)+A y^p\leq B,\\
			&y(0)=y_0,
		\end{array}
		\right.
	\end{align*}
	with some constants $A>0$, $B\geq 0$ and $p\geq 1$. Then for $t>0$, we have 
	\begin{align*}
		y(t)\leq \max\left\{y_0, \: \left(\frac{B}{A}\right)^\frac{1}{p}\right\}.
	\end{align*}
\end{lemma}

\begin{lemma}
	The classical solution $(\uo,\ut,\vo, \vt)$ of \eqref{1.1} satisfies
	\begin{align}
		\ints u_1&\leq M_1 :=\mathrm{max} \left\{\ints u_{10}+\frac{\eta_1}{\eta_2}\ints u_{20}, \frac{|\Omega|}{4}\left(\frac{(\sio+1)^2}{a_1}+\frac{\eta_1}{\eta_2}\frac{(\sit+1)^2}{a_2}\right)\right\}, \label{l1.3}\\
		\ints u_2&\leq M_2 := \mathrm{max}\bigg\{\ints u_{20},\momeg\bigg\},\label{l1.4}
	\end{align}
	for all $t\in(0, \tmax)$.
\end{lemma}
\begin{proof}
	Proof of this lemma is similar to \cite{gnanasekaran2024}. 
	From second equation of \eqref{1.1}, we see that
	\begin{align*}
		\dt\ints\ut\leq\sit\ints\ut-\lambda_2\ints\ut^2.
	\end{align*}
	Applying the Cauchy-Schwarz inequality,
	\begin{align*}
		\dt\ints\ut\leq\sit\ints\ut-\frac{\lambda_2}{\momeg}\left(\ints\ut\right)^2,
	\end{align*}
	by Lemma \ref{l.2.1} yields \eqref{l1.4}. Subsequently, the first equation and $\displaystyle{\frac{\eta_1}{\eta_2}}$ times the second equation of \eqref{1.1} are ntegrated and summed up on $\Omega$ yields
	\begin{align*}
		\dt\left(\ints\uo+\ints\frac{\eta_1}{\eta_2}\ut\right)\leq&\:\sio\ints\uo-\lambda_1\ints\uo^2+\eta_1\ints\uo\ut+\frac{\eta_1\sit}{\eta_2}\ints\ut-\frac{\eta_1\lambda_2}{\eta_2}\ints\ut^2\\
		&-\eta_1\ints\uo\ut,
	\end{align*}
	Appending terms on both sides, we get
	\begin{align*}
		\dt\left(\ints\uo+\ints\frac{\eta_1}{\eta_2}\ut\right)+\left(\ints\uo+\ints\frac{\eta_1}{\eta_2}\ut\right)\leq&\: (\sio+1)\ints\uo-\lambda_1\ints\uo^2+\frac{\eta_1}{\eta_2}(\sit+1)\ints\ut\\
		&-\frac{\eta_1\lambda_2}{\eta_2}\ints\ut^2.
	\end{align*}
	Using the Cauchy's inequality, the equation takes the form
	\begin{align*}
		\dt\left(\ints\uo+\ints\frac{\eta_1}{\eta_2}\ut\right)+\left(\ints\uo+\ints\frac{\eta_1}{\eta_2}\ut\right)\leq&\: \frac{|\Omega|}{4}\left(\frac{(\sio+1)^2}{\lambda_1}+\frac{\eta_1}{\eta_2\lambda_2}(\sit+1)^2\right).
	\end{align*}
	Set $\displaystyle{X(t)= \ints\uo+\ints\frac{\eta_1}{\eta_2}\ut}$, the above inequality can be written as $X'(t)+X(t)\leq C$, where\\
	$\displaystyle{C=\frac{|\Omega|}{4}\left(\frac{(\sio+1)^2}{\lambda_1}+\frac{\eta_1}{\eta_2\lambda_2}(\sit+1)^2\right)}$. Utilizing the ODE argument Lemma \ref{l.2.1}, finally yields \eqref{l1.3}.
\end{proof}

\begin{lemma}[Maximal Sobolev regularity \cite{cao2014, hieber1997}]\label{l3}
	Let $r\in(1,\infty)$ and $T\in(0,\infty)$. Consider the following evolution equation
	\begin{align*}
		\left\{
		\begin{array}{rrll}
			\hspace*{-0.5cm}&&y_t=\Delta y- y+g,\hspace*{2cm}& x\in \Omega,\: t>0,\\
			\hspace*{-0.5cm}&&\frac{\partial y}{\partial\nu}=0,&x\in \partial\Omega,\: t>0,\\
			\hspace*{-0.5cm}&&y(x,0)=y_0(x), &x\in\Omega.
		\end{array}
		\right.
	\end{align*}
	For each $y_0\in{\mathcal{W}^{2,r}}(\Omega) \:(r>n)$ with $\frac{\partial y_0}{\partial\nu}=0$ on $\partial\Omega$ and any $g\in {\mathcal{L}^{r}}((0,T);{\mathcal{L}^{r}}(\Omega))$, there exists a unique solution
	\begin{align*}
		y \in {\mathcal{W}^{1,r}}\left((0,T);{\mathcal{L}^r}(\Omega)\right)\cap{\mathcal{L}^r}\left((0,T); {\mathcal{W}^{2,r}}(\Omega)\right).
	\end{align*}
	Moreover, there exists $C_r>0$, such that 
	\begin{align*}
		\int_0^T \|y(\cdot,t)\|^r_{\mathcal{L}^r(\Omega)}\mathrm{d}t+&\int_0^T\|y_t(\cdot,t)\|^r_{\mathcal{L}^r(\Omega)}\mathrm{d}t+\int_0^T\|\Delta y(\cdot,t)\|^r_{\mathcal{L}^r(\Omega)}\mathrm{d}t\\
		\leq& C_r\int_0^T\|g(\cdot,t)\|^r_{\mathcal{L}^r(\Omega)}\mathrm{d}t+C_r\|y_0\|^r_{\mathcal{L}^r(\Omega)}+C_r\|\Delta y_0\|^r_{\mathcal{L}^r(\Omega)}.
	\end{align*}
	If $s_0\in[0,T)$ and $y(\cdot, s_0)\in \mathcal{W}^{2,r}(\Omega) \:(r>n)$ with $\frac{\partial y(\cdot, s_0)}{\partial\nu}=0$ on $\partial\Omega$, then
	\begin{align*}
		\intT e^{sr}\|\Delta y(\cdot,t)\|^r_{\mathcal{L}^r(\Omega)}\mathrm{d}t\leq C_r\intT e^{sr}\|g(\cdot,t)\|^r_{\mathcal{L}^r(\Omega)}\mathrm{d}t+C_r \|y(\cdot
		, s_0)\|^r_{\mathcal{L}^r(\Omega)}+C_r\|\Delta y(\cdot, s_0)\|^r_{\mathcal{L}^r(\Omega)}.
	\end{align*}
\end{lemma}

Next we prove the main result of our problem \eqref{1.1}.

\section{Global existence of solutions}
\quad This section is dedicated to demonstrating the global existence and boundedness of the solution to \eqref{1.1}. First we derive $\lps$ bound for $\uo$ and $\ut$, $t\in (s_0,\tmax)$.
\begin{lemma}\label{l3.1}
	Suppose that $\Omega\subset\rsn (n\geq 2)$ is a bounded domain with smooth boundary. Assume that for any $q> 1$, there exists $\lambda(q,d_{12},d_{22},a_1,a_2,\eta_1)>0$, such that if $\lambda<\min\left\{\frac{\lambda_1}{2}, \frac{\lambda_2}{2}\right\}$, then
	\begin{align*}
		\nuo_{\lps}+\nut_{\lps}\leq C, \qquad \qquad \forall \, t\in(0,\tmax),
	\end{align*}
	for some $C>0$.
\end{lemma}
\begin{proof}
	Multiplying the first equation of \eqref{1.1} by $\uopo$, $q>1$ and integrating   over $\Omega$, we get
	\begin{align*}
		\ints u_{1t}\uopo&=d_{11}\ints\uopo\Delta\uo+d_{12}\ints\uopo\lvo+\ints\uopo\uo(\sio-\lambda_1\uo+\eta_1\ut).
	\end{align*}
	Applying the technique of integration by parts, 
	\begin{align}
		\fp\dt\ints\uop=&-d_1(q-1)\ints\uopt|\guo|^2+d_{12}\ints\uopo\lvo+\sio\ints\uop\nonumber\\
		&-\lambda_1\ints\uopp+\eta_1\ints\uop\ut.\label{l4.1.1}
	\end{align}
	Utilizing the Gagliardo-Nirenberg inequality  and the Young's inequality, it follows that
	\begin{align*}
		\ints\uop=\left\|\uo^{\frac{q}{2}}\right\|^2_{\lts}&\leq\: c'_1\left(\left\|\nabla\uo^{\frac{q}{2}}\right\|^{2b}_{\lts}\,\,\left\|\uo^{\frac{q}{2}}\right\|^{2(1-b)}_{\ltps}+\left\|\uo^{\frac{q}{2}}\right\|^2_{\ltps}\right)\nonumber\\
		&\leq \frac{4d_1(q-1)}{q(q+1)}\left(\left\|\nabla\uo^{\frac{q}{2}}\right\|^{2b}_{\lts}\right)^\frac{1}{b}+c'_2\left(\left\|\uo^{\frac{q}{2}}\right\|^{2(1-b)}_{\ltps}\right)^\frac{1}{1-b}+c'_1\left\|\uo^{\frac{q}{2}}\right\|^2_{\ltps}\nonumber\\
		&\leq \frac{4d_1(q-1)}{q(q+1)}\left\|\nabla\uo^{\frac{q}{2}}\right\|^{2}_{\lts}+c'_3 \nuo^q_{\los}\nonumber\\
		&\leq \frac{4d_1(q-1)}{q(q+1)}\frac{q^2}{4}\ints\uo^{q-2}\mguo^2+c'_3 M^q_1.\nonumber
	\end{align*}
	Hence we obtain
	\begin{align}
		\fpop\ints\uop&\leq d_1(q-1)\ints \uopt\mguo^2+c_1,\label{l4.1.2}
	\end{align} 
	with $c_1>0$,
	where $b=\frac{\frac{q}{2}-\frac{1}{2}}{\frac{q}{2}+\frac{1}{n}-\frac{1}{2}}\in(0,1)$. 
	Now we rewrite the above equation \eqref{l4.1.2} as follows 
	\begin{align}
		-d_1(q-1)\ints \uopt\big|\guo\big|^2\leq -\fpop\ints\uop+c_1.\label{l4.1.3}
	\end{align}
	Using Young's inequality two times to the second term in \eqref{l4.1.1}, we get
	\begin{align}
		d_{12}\ints\uopo\lvo\leq \frac{\lambda_1}{6}\ints\uopp+c'_4\ints\mlv^\frac{q+1}{2}\leq \frac{\lambda_1}{6}\ints\uopp+c'_5 \ints\mlv^{q+1}+c_2.\label{14.1.4}
	\end{align}
	where $c_2>0$. Again, using the Young's inequality to the third term in \eqref{l4.1.1}, we obtain
	\begin{align}
		\sio\ints\uop
		\leq \frac{\lambda_1}{6}\ints \uopp+c_3,\label{l4.1.5}
	\end{align}
	where $c_3>0$, and for fifth term in \eqref{l4.1.1} as follows
	\begin{align}
		\ints \eta_1\uop\ut
		\leq \frac{\lambda_1}{6}\ints\uopp+c_4\ints \utpp,\label{l4.1.6}
	\end{align}
	with $c_4>0$.
	Substituting \eqref{l4.1.3} - \eqref{l4.1.6} in \eqref{l4.1.1}, we see that
	\begin{align*}
		\dt\left(\fp\ints\uop\right)+(q+1)\left(\fp\ints\uop\right)\leq & -\frac{\lambda_1}{2}\ints\uopp+c'_5\ints\mlvop+c_4\ints \utpp+c_5.
	\end{align*}
	Adopting the variation of constants formula, we get
	\begin{align}
		\fp\ints\uop \leq &\: -\frac{\lambda_1}{2}\intts\ints\epts\uopp+c'_5\intts\ints\epts\mlvop\nonumber\\
		&+c_4\intts\ints\epts \utpp+c_6,\label{l4.1.7}
	\end{align}
	where $c_6>0$.
	According to Lemma \ref{l3}, there
	exists $c_7 > 0$ such that
	\begin{align}
		c'_5\intts\ints \epts\mlvop\leq&\: c_7\intts\ints \epts\ut^{q+1}\nonumber\\
		&+c_7\big\|\vo(\cdot,s_0)\big\|^{q+1}_{\mathcal{L}^{q+1}(\Omega)}+c_7\big\|\lvo(\cdot,s_0)\big\|^{q+1}_{\mathcal{L}^{q+1}(\Omega)}.\label{l4.1.8}
	\end{align}
	Substituting the last inequality \eqref{l4.1.8} into \eqref{l4.1.7}, we arrive at
	\begin{align}
		\fp\ints\uop \leq & -\frac{\lambda_1}{2}\intts\ints\epts\uopp+c_8\intts\ints\epts\utpp+c_9,\label{l4.1.9}
	\end{align}
	where, $c_9>0$.
	Similarly, we estimate for $\ut$ as
	\begin{align}
		\fp\ints\utp \leq & -\frac{\lambda_2}{2}\intts\ints\epts\utpp+c_{10}\intts\ints\epts\uopp+c_{11},\label{l4.1.10}
	\end{align}
	where the constant $c_{11}>0$. Adding \eqref{l4.1.9} and \eqref{l4.1.10} affords us
	\begin{align}
		\fp\left(\ints\uop+\ints\utp\right) \leq & -\left(\frac{\lambda_1}{2}-c_{10}\right)\intts\ints\epts\uopp\nonumber\\
		&-\left(\frac{\lambda_2}{2}-c_8\right)\intts\ints\epts\utpp+c_{12}.
		\label{l4.1.11}
	\end{align}
	Let $\lambda=\mbox{max}\{c_8, c_{10}\}$ such that $0<\lambda<\min\left\{\frac{\lambda_1}{2}, \frac{\lambda_2}{2}\right\}$. Hence, we deduce from \eqref{l4.1.11} that
	\begin{align*}
		\ints\uop+\ints\utp&\leq c_{12},\qquad \forall\, t\in(0,\tmax),
	\end{align*}
	where the constant $c_{12}>0$. In view of \eqref{l3.0}, the proof is complete.
\end{proof}

{\bf Proof of Theorem \ref{t1}.} 
Assume that $\tmax<\infty$. The result follows from the standard parabolic regularity (Ladyzhenskaya et al. \cite{ladyzen}, Amann \cite{amann1995}) applied to the third equation in \eqref{1.1},
\begin{align*}
	v_{1t}-d_3\Delta v_1+b_1v_1=a_1 u_2
\end{align*}
which ensures the boundedness and regularity of $v_1$, provided that $u_2$ satisfies the $\lps$-bound obtained in Lemma \ref{l3.1}, and similar procedure for $v_2$. Hence, there exists a constant $C_1>0$ such that
\begin{align}
	\nvo_{\mathcal{W}^{1,\infty}(\Omega)}+\nvt_{\mathcal{W}^{1,\infty}(\Omega)}\leq C_1, \qquad \forall t\in(0, \tmax).\label{t.1.1.1}
\end{align}
Finally, one can employ the well-known Moser-Alikakos iteration technique (\cite{tao2012} Lemma A.1) with Lemma \ref{l3.1} to prove that there exists $C_2>0$ fulflling
\begin{align*}
	\nuo_{\lis}+\nut_{\lis}\leq C_2, \qquad \forall t\in(0, \tmax).
\end{align*}
This bound contradicts \eqref{l2.2.1}. Hence, it follows that $\tmax = \infty$. The proof is complete.

\section{Global asymptotic stability}
\quad In this section, we examine the global asymptotic behaviour of solutions and equilibrium convergence rates to the system \eqref{1.1}, utilizing the Lyapunov functional. The proof of these asymptotic behavior outcomes takes inspiration from the research presented in \cite{xbai}. The following lemma's proof is derived from \cite{mwang2016} and \cite{jwang2018}.
\begin{lemma}\label{l.4.1}
	Let $(u_1, u_2, v_1, v_2)$ be the nonnegative classical solution of the system \eqref{1.1} and suppose that the assumptions of Theorem \ref{t1} hold true. Then there exists $\theta\in(0, 1)$ and $C>0$ such that
	\begin{align*}
		\Big\|u_1\Big\|_{\mathcal{C}^{2+\theta, 1+ \frac{\theta}{2}}(\overline{\Omega}\times[t, t+1])}+\Big\|u_2\Big\|_{\mathcal{C}^{2+\theta, 1+ \frac{\theta}{2}}(\overline{\Omega}\times[t, t+1])}\leq C
	\end{align*}
	and
	\begin{align*}
		\Big\|v_1\Big\|_{\mathcal{C}^{2+\theta, 1+ \frac{\theta}{2}}(\overline{\Omega}\times[t, t+1])}+\Big\|v_2\Big\|_{\mathcal{C}^{2+\theta, 1+ \frac{\theta}{2}}(\overline{\Omega}\times[t, t+1])}\leq C,
	\end{align*}
	for all $t\geq 1$.
\end{lemma}
\begin{proof}
	The proof is based on the standard parabolic regularity theory in \cite{ladyzen} and Theorem \ref{t1}. For more details see, \cite{tli, mmporzio}
\end{proof}

\begin{lemma}[\cite{xbai}]\label{l.4.2}
	Suppose that $f:(1, \infty)$ is a uniformly continuous nonnegative function such that 
	\begin{align*}
		\int_{1}^\infty f(t)d\mathrm{t}<\infty.
	\end{align*} 
	Then, $f(t)\to 0$ as $t\to\infty$.
\end{lemma}

First we start with the coexistent state of the species.

\subsection{Coexistence state of the species}
Here we assume that $\eta_2<\frac{\sit \lambda_1}{\sio}$	and
\begin{align*}
	d_{12}^2<\frac{16d_{11} d_3 b_1\eta_1\lambda_2(\lambda_1\lambda_2+\eta_1\eta_2)\nuo^2_{\lis}}{a_1^2\eta_2(\lambda_2\sio+\eta_1\sit)},\quad 
	d_{22}^2<\frac{16d_{21}d_4b_2\eta_2\lambda_1(\lambda_1\lambda_2+\eta_1\eta_2)\nut^2_{\lis}}{a_2^2\eta_1(\lambda_1\sit-\eta_2\sio)}
\end{align*}
hold. Let $(u_1, u_2, v_1, v_2)$ be the classical solution of \eqref{1.1} satisfying \eqref{1.2} and $(\uost, \utst, \vost, \vtst)$ be the unique positive equilibrium point of the system \eqref{1.1}.

\begin{lemma}\label{l6}
	There exist $\delta_1, \delta_2>0$ and $\epsilon_1>0$ such that the functions 
	\begin{align*}
		\mathscr{E}_1(t)=&\ints\left(\uo-\uost-\uost \ln\frac{\uo}{\uost}\right)+\frac{\eta_1}{\eta_2}\ints\left(\ut-\utst-\utst \ln\frac{\ut}{\utst}\right)+\frac{\delta_1}{2}\ints\left(\vo-\vost\right)^2\nonumber\\
		&+\frac{\delta_2}{2}\ints\left(\vt-\vtst\right)^2
	\end{align*}
	and
	\begin{align*}
		f_1(t)=&\ints(\uo-\uost)^2+\ints(\ut-\utst)^2+\ints(\vo-\vost)^2+\ints(\vt-\vtst)^2
	\end{align*}
	satisfy
	\begin{align}
		\dt\mathscr{E}_1(t)\leq -\epsilon_1 f_1(t), \quad t>0.\label{l6.4}
	\end{align}
	where $\epsilon_1>0$.
\end{lemma}
\begin{proof}
	Choose $\Gamma_1\in(0, \lambda_2)$, $\Gamma_2 \in (0, \lambda_1)$ and fix 
	$\delta_1\in \left(\frac{d_{12}^2\uost}{4d_{11}d_3\nuo^2_{\lis}},\: \frac{4b_1 \eta_1\lambda_2}{a_1^2\eta_2}\right)$, \\
	$\delta_2\in \left(\frac{d_{22}^2\utst \eta_1}{4d_{21}d_4 \eta_2\nut^2_{\lis}},\: \frac{4b_2a_1}{a_2^2}\right)$.
	Let us consider the energy of the system as follows
	\begin{align*}
		\mathscr{E}_1(t)=&\mathscr{A}_1(t)+\frac{\eta_1}{\eta_2}\mathscr{B}_1(t)+\mathscr{C}_1(t)+\mathscr{D}_1(t), \quad t>0,
	\end{align*}
	Now, let $\mathscr{H}(u)=u-u^e \ln u$, for $u>0$. Using the Taylor's formula, we have
	\begin{align*}
		\mathscr{H}(u)=\mathscr{H}(u^e)+\mathscr{H'}(u^e)(u-u^e)+\frac{1}{2}\mathscr{H''}(u^e)(u-u^e)^2
		=\frac{1}{2u^e}(u-u^e)^2\geq 0.
	\end{align*}
	Therefore, we have $\mathscr{A}_1(t)=\ints \mathscr{H}(\uo)-\mathscr{H}(\uo^e)\geq 0$. Similarly, we can obtain $\mathscr{B}_1(t)\geq 0$ and $\mathscr{C}_1(t), \mathscr{D}_1(t)\geq 0$ by the nonnegativity of $\delta_i, i=1,2$.\\
	Now, we can  compute
	\begin{align*}
		\dt\mathscr{A}_1(t)
		=&-d_{11}\uost\ints\left|\frac{\guo}{\uo}\right|^2-d_{12}\uost\ints\frac{\guo}{\uo^2}\cdot\gvo+\ints(\uo-\uost)(\sio-\lambda_1\uo+\eta_1\ut)\\
		=&-d_{11}\uost\ints\left|\frac{\guo}{\uo}\right|^2-d_{12}\uost\ints\frac{\guo}{\uo^2}\cdot\gvo-\lambda_1\ints(\uo-\uost)^2\\
		&+\eta_1\ints(\uo-\uost)(\ut-\utst).
	\end{align*}
	Analogously, for $\mathscr{B}_1(t)$,
	\begin{align*}
		\dt\mathscr{B}_1(t)
		=&-d_{21}\utst\ints\left|\frac{\gut}{\ut}\right|^2-d_{22}\utst\ints\frac{\gut}{\ut^2}\cdot\gvt+\ints(\ut-\utst)(\sit-\lambda_2\ut-\eta_2\uo)\\
		=&-d_{21}\utst\ints\left|\frac{\gut}{\ut}\right|^2-d_{22}\utst\ints\frac{\gut}{\ut^2}\cdot\gvt-\lambda_2\ints(\ut-\utst)^2\\
		&-\eta_2\ints(\uo-\uost)(\ut-\utst).
	\end{align*}
	Let us now consider
	\begin{align*}
		\dt\mathscr{C}_1(t)
		=&-d_3\delta_1\ints|\gvo|^2+a_1\delta_1\ints(\vo-\vost)\ut+a_1\delta_1\ints(\vo-\vost)\utst-a_1\delta_1\ints(\vo-\vost)\utst\\
		&-b_1\delta_1\ints(\vo-\vost)\vo+b_1\delta_1\ints(\vo-\vost)\vost-b_1\delta_1\ints(\vo-\vost)\vost\\
		=&-d_3\delta_1\ints|\gvo|^2+a_1\delta_1\ints(\vo-\vost)(\ut-\utst)-b_1\delta_1\ints(\vo-\vost)^2.
	\end{align*}
	By the similar way, we obtain
	\begin{align*}
		\dt\mathscr{D}_1(t)
		=&-d_4\delta_2\ints|\gvt|^2+a_2\delta_2\ints(\vt-\vtst)(\uo-\uost)-b_2\delta_2\ints(\vt-\vtst)^2.
	\end{align*}
	Combining all the terms, we get
	\begin{align}
		\dt\mathscr{E}_1(t)=&-d_{11}\uost\ints\left|\frac{\guo}{\uo}\right|^2-d_{12}\uost\ints\frac{\guo}{\uo^2}\cdot\gvo-\lambda_1\ints(\uo-\uost)^2\nonumber\\
		&+\eta_1\ints(\uo-\uost)(\ut-\utst)-d_{21}\frac{\eta_1}{\eta_2}\utst\ints\left|\frac{\gut}{\ut}\right|^2-d_{22}\frac{\eta_1}{\eta_2}\utst\ints\frac{\gut}{\ut^2}\cdot\gvt\nonumber\\
		&-\frac{\eta_1\lambda_2}{\eta_2}\ints(\ut-\utst)^2-\eta_1\ints(\uo-\uost)(\ut-\utst)-d_3\delta_1\ints|\gvo|^2\nonumber\\
		&+a_1\delta_1\ints(\vo-\vost)(\ut-\utst)-b_1\delta_1\ints(\vo-\vost)^2-d_4\delta_2\ints|\gvt|^2\nonumber\\
		&+a_2\delta_2\ints(\vt-\vtst)(\uo-\uost)-b_2\delta_2\ints(\vt-\vtst)^2.\label{l6.6}
	\end{align}
	Hence, by the Cauchy's inequality, we have acquired the following estimates
	\begin{align}
		-d_{12}\uost\ints\frac{\guo}{\uo}\cdot\frac{\gvo}{\uo}\leq&\: d_{11}\uost \ints\left|\frac{\guo}{\uo}\right|^2+\frac{d_{12}^2\uost}{4d_{11}\nuo^2_{\lis}}\ints|\gvo|^2,\label{l6.7}\\
		-d_{22}\frac{\eta_1}{\eta_2}\utst\ints\frac{\gut}{\ut}\cdot\frac{\gvt}{\ut}\leq&\: d_{21}\frac{\eta_1}{\eta_2}\utst \ints\left|\frac{\gut}{\ut}\right|^2+\frac{d_{22}^2\eta_1\utst}{4d_{21}\eta_2\nut^2_{\lis}}\ints|\gvt|^2,\label{l6.8}\\
		a_1\delta_1\ints(\vo-\vost)(\ut-\utst)\leq&\: \Gamma_1 \frac{\eta_1}{\eta_2} \ints(\ut-\utst)^2+\frac{a_1^2\delta_1^2\eta_2}{4\Gamma_1\eta_1}\ints(\vo-\vost)^2,\label{l6.9}\\
		a_2\delta_2\ints(\vt-\vtst)(\uo-\uost)\leq&\: \Gamma_2 \ints(\uo-\uost)^2+\frac{a_2^2\delta_2^2}{4\Gamma_2}\ints(\vt-\vtst)^2.\label{l6.10}
	\end{align}
	Substituting the above estimates \eqref{l6.7}-\eqref{l6.10} in \eqref{l6.6}, we obtain
	\begin{align*}
		\dt\mathscr{E}_1(t)
		\leq&-(\lambda_1-\Gamma_2)\ints(\uo-\uost)^2-\frac{\eta_1}{\eta_2}(\lambda_2-\Gamma_1)\ints(\ut-\utst)^2\\
		&-\left(b_1\delta_1-\frac{a_1^2\delta_1^2\eta_2}{4\Gamma_1\eta_1}\right)\ints(\vo-\vost)^2-\left(b_2\delta_2-\frac{a_2^2\delta_2^2}{4\Gamma_2}\right)\ints(\vt-\vtst)^2\\
		&-\left(d_3\delta_1-\frac{d_{12}^2\uost}{4d_{11}\nuo^2_{\lis}}\right)\ints|\gvo|^2-\left(d_4\delta_2-\frac{d_{22}^2\eta_1\utst}{4d_{21}\eta_2\nut^2_{\lis}}\right)\ints|\gvt|^2.
	\end{align*}
	Therefore, we get
	\begin{align*}
		\dt\mathscr{E}_1(t)\leq -\epsilon_1&\left( \ints(\uo-\uost)^2+\ints(\ut-\utst)^2+\ints(\vo-\vost)^2+\ints(\vt-\vtst)^2\right),
	\end{align*}
	we arrive at \eqref{l6.4} with $\epsilon_1>0$.
\end{proof}

\begin{lemma}\label{l7}
	Suppose $(\uost, \utst, \vost, \vtst)$ be the coexistence state of \eqref{1.1}, then the following asymptotic behavior holds
	\begin{align}
		\big\|\uot-\uost\big\|_{\lis}+\big\|\utt-\utst\big\|_{\lis}+\big\|\vot-\vost\big\|_{\lis}+\big\|\vtt-\vtst\big\|_{\lis}\to 0\label{l7.1}
	\end{align}
	as $t\to\infty$.
\end{lemma}
\begin{proof}
	From \eqref{l6.4}, we have
	\begin{align*}
		\dt\mathscr{E}_1(t)\leq -\epsilon_1 f_1(t), \quad t>0.
	\end{align*}
	Integrating over $t$, we get
	\begin{align*}
		\int_1^\infty f_1(t)\leq \frac{1}{\epsilon_1}\Big(\mathscr{E}_1(1)-\mathscr{E}_1(t)\Big)<\infty.
	\end{align*}
	Lemma \ref{l.4.1} implies that $\uo$, $\ut$, $\vo$, and $\vt$ are uniformly H\"older continuous in $\overline{\Omega}\times[t, t+1]$ with respect to $t$. Using Theorem \ref{t1}, we establish that $(u_1(\cdot,t))_{t>1}$ is bounded in $\mathcal{W}^{1,\infty}(\Omega)$. Therefore, $f_1(t)$ is uniformly continuous in $(1, \infty)$, as stated in Lemma \eqref{l.4.2}, which yields
	\begin{align*}
		\ints(\uot-\uost)^2+\ints(\utt-\utst)^2+\ints(\vot-\vost)^2+\ints(\vtt-\vtst)^2\to 0
	\end{align*}
	as $t\to\infty$. Applying the Gagliardo-Nirenberg inequality, we have
	\begin{align}
		\Big\|\uot-\uost\Big\|_{\lis}\leq C_1\Big\|\uot-\uost\Big\|^{\frac{n}{n+2}}_{\mathcal{ W}^{1,\infty}(\Omega)}\:\:\Big\|\uot-\uost\Big\|^{\frac{2}{n+2}}_{\lts}, \quad t>0.\label{l7.2}
	\end{align}
	As a consequence, we can infer that $\uot$ converges to $\uost$ in $\lis$ when $t$ tends to infinity. Similarly, through analogous reasoning, we can obtain \eqref{l7.1}.
\end{proof}

\subsection{Prey vanishing state}
Here we assume that $\displaystyle{\eta_2\geq\frac{\sit \lambda_1}{\sio}}$ and
\begin{align*}
	\displaystyle{d_{12}^2<\frac{16d_{11} d_3b_1\eta_1\lambda_1\lambda_2\nuo_{\lis}^2}{a_1^2\sio \eta_2}}
\end{align*} hold. Let $(u_1, u_2, v_1, v_2)$ be the classical solution of \eqref{1.1} satisfying \eqref{1.2}. The proof is similar to the previous Lemmas \ref{l6} and \ref{l7}.
\begin{lemma}\label{l9}
	There exists $\delta_2, \delta_3>0$ and $\epsilon_2>0$ such that the functions 
	\begin{align*}
		\mathscr{E}_2(t)=&\ints\left(\uo-\frac{\sio}{\lambda_1}-\frac{\sio}{\lambda_1}\ln\frac{\uo}{\sio}\right)+\frac{\eta_1}{\eta_2}\ints\ut+\frac{\delta_3}{2}\ints\vo^2+\frac{\delta_4}{2}\ints\left(\vt-\frac{\sio a_2}{\lambda_1b_2}\right)^2
	\end{align*}
	and
	\begin{align*}
		f_2(t)=&\ints\left(\uo-\frac{\sio}{\lambda_1}\right)^2+\ints\ut^2+\ints\vo^2+\ints\left(\vt-\frac{\sio a_2}{\lambda_1b_2}\right)^2
	\end{align*}
	satisfy
	\begin{align}
		\dt\mathscr{E}_2(t)\leq -\epsilon_2 f_2(t)-\left(\frac{\eta_1\sio}{\lambda_1}-\frac{\eta_1\sit}{\eta_2}\right)\ints\ut, \quad t>0,\label{l9.4}
	\end{align}
	where $\epsilon_2>0$.
\end{lemma}
\begin{proof}
	Choose $\eta_3\in(0, \delta_3b_1)$ and $\eta_4\in(0, \lambda_1)$ and fix $\displaystyle{\delta_3\in\left(\frac{d_{12}^2\sio}{4d_{11}d_3\lambda_1\nuo_{\lis}^2}, \frac{4b_1 \eta_1\lambda_2}{a_1^2 \eta_2}\right)}$ \\
	and $\delta_4\in \displaystyle{\left(0, \frac{4b_2\lambda_1}{a_2^2}\right)}$. Consider the energy functional as follows
	\begin{align*}
		\mathscr{E}_2(t)=&\mathscr{A}_2(t)+\frac{\eta_1}{\eta_2}\mathscr{B}_2(t)+\mathscr{C}_2(t)+\mathscr{D}_2(t), \quad t>0,
	\end{align*}
	A simple computation gives
	\begin{align*}
		\dt\mathscr{A}_2(t)=&-\frac{d_{11}\sio}{\lambda_1}\ints\left|\frac{\guo}{\uo}\right|^2-\frac{d_{12}\sio}{\lambda_1}\ints\frac{\guo}{\uo^2}\cdot\gvo+\frac{1}{\lambda_1}\ints(\lambda_1\uo-\sio)(\sio-\lambda_1\uo+\eta_1\ut)\\
		=&\frac{d_{12}^2\sio}{4d_{11}\lambda_1\nuo_{\lis}^2}\ints|\gvo|^2+\frac{1}{\lambda_1}\ints(\lambda_1\uo-\sio)(\sio-\lambda_1\uo+\eta_1\ut)
	\end{align*}
	and
	\begin{align*}
		\dt\mathscr{B}_2(t)=& \ints\ut(\sit-\lambda_2\ut-\eta_2\uo).
	\end{align*}
	Similarly, we get
	\begin{align*}
		\dt\mathscr{C}_2(t)=& -d_3\delta_3\ints|\gvo|^2+a_1\delta_3\ints\vo\ut-\delta_3b_1\ints\vo^2.
	\end{align*}
	Employing Cauchy's inequality, the second expression on the right-hand side of the previous estimation yields.
	\begin{align*}
		\dt\mathscr{C}_2(t)\leq&-d_3\delta_3\ints|\gvo|^2+\eta_3\ints\vo^2+\frac{a_1^2\delta_3^2}{4\eta_3}\ints\ut^2-\delta_3b_1\ints\vo^2.
	\end{align*}
	By a similar fashion, we obtain
	\begin{align*}
		\dt\mathscr{D}_2(t)=&-d_4\delta_4\ints|\gvt|^2+a_2\delta_4\ints\left(\vt-\frac{\sio a_2}{\lambda_1b_2}\right)\uo-b_2\delta_4\ints\left(\vt-\frac{\sio a_2}{\lambda_1b_2}\right)\vt\\
		=&-d_4\delta_4\ints|\gvt|^2+a_2\delta_4\ints\left(\vt-\frac{\sio a_2}{\lambda_1b_2}\right)\uo-b_2\delta_4\ints\left(\vt-\frac{\sio a_2}{\lambda_1b_2}\right)\vt\\
		&+b_2\delta_4\ints\left(\vt-\frac{\sio a_2}{\lambda_1b_2}\right)\frac{\sio a_2}{\lambda_1b_2}-b_2\delta_4\ints\left(\vt-\frac{\sio a_2}{\lambda_1b_2}\right)\frac{\sio a_2}{\lambda_1b_2}\\
		=&-d_4\delta_4\ints|\gvt|^2+a_2\delta_4\ints\left(\vt-\frac{\sio a_2}{\lambda_1b_2}\right)\left(\uo-\frac{\sio}{\lambda_1}\right)-b_2\delta_4\ints\left(\vt-\frac{\sio a_2}{\lambda_1b_2}\right)^2.
	\end{align*}
	By making use of Cauchy's inequality, the second term on the right-hand side of the aforementioned estimation yields
	\begin{align*}
		\dt\mathscr{D}_2(t)\leq&-d_4\delta_4\ints|\gvt|^2+\eta_4\ints\left(\uo-\frac{\sio}{\lambda_1}\right)^2+\frac{a_2^2\delta_4^2}{4\eta_4}\ints\left(\vt-\frac{\sio a_2}{\lambda_1b_2}\right)^2\\
		&-b_2\delta_4\ints\left(\vt-\frac{\sio a_2}{\lambda_1b_2}\right)^2.
	\end{align*}
	By consolidating all the terms, it is possible to derive.
	\begin{align*}
		\dt\mathscr{E}_2(t)
		\leq&\:-\left(\frac{\eta_1\sio}{\lambda_1}-\frac{\eta_1\sit}{\eta_2}\right)\ints\ut-(\lambda_1-\eta_4)\ints\left(\uo-\frac{\sio}{\lambda_1}\right)^2-\left(\frac{\eta_1\lambda_2}{\eta_2}-\frac{a_1^2\delta_3^2}{4\eta_3}\right)\ints\ut^2\\
		&-(\delta_3b_1-\eta_3)\ints\vo^2-\left(b_2\delta_4-\frac{a_2^2\delta_4^2}{4\eta_4}\right)\ints\left(\vt-\frac{\sio a_2}{\lambda_1b_2}\right)^2\\
		&-\left(d_3\delta_3-\frac{d_{12}^2\sio}{4d_{11}\lambda_1\nuo_{\lis}^2}\right)\ints|\gvo|^2-d_4\delta_4\ints|\gvt|^2,
	\end{align*}
	which implies \eqref{l9.4} with $\epsilon_2>0$.
\end{proof}

\begin{lemma}\label{l10}
	Suppose $\left(\frac{\sio}{\lambda_1}, 0, 0, \frac{\sio a_2}{\lambda_1b_2}\right)$ be the positive semi-trivial equilibrium of \eqref{1.1}, then the following asymptotic behavior holds
	\begin{align*}
		\Big\|\uot-\frac{\sio}{\lambda_1}\Big\|_{\lis}+\Big\|\utt\Big\|_{\lis}+\Big\|\vot\Big\|_{\lis}+\Big\|\vtt-\frac{\sio a_2}{\lambda_1b_2}\Big\|_{\lis}\to 0
	\end{align*}
	as $t\to\infty$.
\end{lemma}
\begin{proof}
	In view of Lemma \ref{l7}, the proof is similar.
\end{proof}

We demonstrate the subsequent theorems using Lemmas \ref{l7}, \ref{l10}, which possess uniform convergence properties.\\\\
{\bf Proof of Theorem \ref{t2}}.
Now, let $\mathscr{H}(u)=u-\uos \ln u$, for $u>0$, from L'H\'ospital's rule, we have 
\begin{align*}
	\lim\limits_{\uo-\uost}\frac{\mathscr{H}(\uo)-\mathscr{H}(\uost)}{(\uo-\uost)^2}=\frac{1-\frac{\uost}{\uo}}{2(\uo-\uost)}=\frac{1}{2\uost}.
\end{align*}
We can thus choose $t_0>0$ and by the Taylor's expansion such that
\begin{align*}
	\ints\left(\uo-\uost-\uost \ln\frac{\uo}{\uost}\right)=\ints(\mathscr{H}(\uo)-\mathscr{H}(\uost))
	\leq\frac{1}{2\uost}\ints(\uo-\uost)^2, \quad \forall t>t_0.
\end{align*}
From the above estimate , we arrive at
\begin{align}
	\frac{1}{4\uost}\ints(\uo-\uost)^2\leq \ints\left(\uo-\uost-\uost \ln\frac{\uo}{\uost}\right)\leq \frac{3}{4\uost}\ints(\uo-\uost)^2, \quad \forall t>t_0.\label{t2.1}
\end{align}
Similarly, we can arrive at
\begin{align}
	\frac{1}{4\utst}\ints(\ut-\utst)^2\leq \ints\left(\ut-\utst-\utst \ln\frac{\ut}{\utst}\right)\leq \frac{3}{4\utst}\ints(\ut-\utst)^2, \quad \forall t>t_0.\label{t2.3}
\end{align}
Considering the right-hand sides of the inequalities \eqref{t2.1} and \eqref{t2.3}, it becomes apparent that for all $t>t_0$ and with $C_1>0$, $\mathscr{E}_1(t)\leq C_1 \mathscr{F}_1(t)$. This result is supported by Lemma \ref{l6}, which states that
\begin{align*}
	\dt\mathscr{E}_1(t)\leq- \epsilon_1 f_1(t)\leq -\frac{\epsilon_1}{C_1}\mathscr{E}_1(t), \qquad t>t_0.
\end{align*}
Solving the above differential inequality, gives
\begin{align*}
	\mathscr{E}_1(t)\leq C_2e^{-\frac{\epsilon_1}{C_1}t}, \qquad t>t_0.
\end{align*}
By utilizing the solution provided earlier, we can derive expressions from the left-hand side of the inequalities \eqref{t2.1} and \eqref{t2.3}, resulting in the following
\begin{align*}
	f_1(t)\leq C_3 \mathscr{E}_1(t)\leq C_4 e^{-\frac{\epsilon_1}{C_1}t}, \qquad t>t_0.
\end{align*}
Using \eqref{l7.2} and Lemma \ref{l.4.1}, finally we arrive at
\begin{align*}
	\big\|\uot-\uost\big\|_{\lis}+\big\|\utt-\utst\big\|_{\lis}+\big\|\vot-\vost\big\|_{\lis}+\big\|\vtt-\vtst\big\|_{\lis}\leq C_6 \displaystyle{e^{-\frac{\epsilon_1}{C_1}t}},
\end{align*}
for all $t>t_0$.\hfill\qedsymbol\\

{\bf Proof of Theorem \ref{t3}}.
Assuming $\eta_2>\frac{\sit \lambda_1}{\sio}$ and choose $t_0>0$, we can use a comparable approach to that used to derive \eqref{t2.1} and \eqref{t2.3} to arrive at the following result
\begin{align}
	\frac{\lambda_1}{4\sio}\ints\left(\uo-\frac{\sio}{\lambda_1}\right)^2\leq \ints\left(\uo-\frac{\sio}{\lambda_1}-\frac{\sio}{\lambda_1} \ln\frac{\uo}{\sio}\right)\leq \frac{3\lambda_1}{4\sio}\ints\left(\uo-\frac{\sio}{\lambda_1}\right)^2, \quad \forall t>t_0.\label{t3.1}
\end{align}
and
\begin{align}
	\frac{1}{2}\ints\ut^2+\frac{1}{2}\ints\ut\leq \ints\ut\leq 2\ints \ut^2+2\ints\ut, \quad \forall t>t_0.\label{t3.2}
\end{align}
Considering the right-hand side of inequalities \eqref{t3.1} and \eqref{t3.2}, it is apparent that $\mathscr{E}_2(t)\leq C_7\left( f_2(t)+\ints\ut\right)$ for all $t>t_0$ with $C_7>0$. In light of Lemma \ref{l9}, with $\left(\frac{\eta_1\sio}{\lambda_1}-\frac{\eta_1\sit}{\eta_2}\right)>\epsilon_2$, we obtain
\begin{align*}
	\dt\mathscr{E}_2(t)
	\leq -\frac{\epsilon_2}{C_7}\mathscr{E}_2(t)-\left(\left(\frac{\eta_1\sio}{\lambda_1}-\frac{\eta_1\sit}{\eta_2}\right)-\epsilon_2\right)\ints\ut
	\leq -\frac{\epsilon_2}{C_7}\mathscr{E}_2(t), \qquad t>t_0,
\end{align*}
such that $\mathscr{E}_2(t)\leq C_8e^{-\frac{\epsilon_2}{C_7}t}$, $t>t_0$. Examining the left-hand side of inequalities \eqref{t3.1} and \eqref{t3.2}, we can see that $f_2(t)\leq C_9 \mathscr{E}_2(t)\leq C_{10}e^{-\frac{\epsilon_2}{C_7}t}$. Hence we get
\begin{align*}
	\Big\|\uot-\sio\Big\|_{\lis}+\Big\|\utt\Big\|_{\lis}+\Big\|\vot\Big\|_{\lis}+\Big\|\vtt-\frac{\sio a_2}{\lambda_1b_2}\Big\|_{\lis}\leq C_{12} e^{-\frac{\epsilon_2}{C_7}t}.
\end{align*}
for all $t>t_0$.

In addition, in the case where $\eta_2=\frac{\sit \lambda_1}{\sio}$, we find that $\left(\frac{\eta_1\sio}{\lambda_1}-\frac{\eta_1\sit}{\eta_2}\right)\ints\ut=0$. Nevertheless, it is worth noting that \eqref{t3.1} and \eqref{t3.2} still remain valid. By utilizing the inequality on the right-hand side of \eqref{t3.1} and applying Cauchy-Schwarz inequality with $C_{13}, C_{14}>0$, we obtain
\begin{align*}
	\mathscr{E}_2(t)\leq&C_{13}\ints\left(\uo-\frac{\sio}{\lambda_1}\right)^2+C_{13}\ints\ut+C_{13}\ints\vo^2+C_{13}\ints\left(\vt-\frac{\sio a_2}{\lambda_1b_2}\right)^2\\
	\leq & C_{14}\left(\ints\left(\uo-\frac{\sio}{\lambda_1}\right)^2\right)^{\frac{1}{2}}+C_{14}\left(\ints\ut^2\right)^{\frac{1}{2}}+C_{14}\left(\ints\vo^2\right)^{\frac{1}{2}}+C_{14}\left(\ints\left(\vt-\frac{\sio a_2}{\lambda_1b_2}\right)^2\right)^{\frac{1}{2}}\\
	\leq & C_{15} f_2^{\frac{1}{2}}(t), \qquad\forall t>t_0,
\end{align*}
such that
\begin{align*}
	\dt\mathscr{E}_2(t)\leq& -\epsilon_2 f_2(t)\leq -\frac{\epsilon_2}{C_{16}}\mathscr{E}_2^2(t),\quad t>t_0.
\end{align*}
Hence, we obtain the solution of the above ODE problem as $\mathscr{E}_2(t)\leq\frac{C_{18}}{\epsilon_2 t+ C_{17}}$. From the left hand side inequalities \eqref{t3.1} and \eqref{t3.2}, we have 
\begin{align*}
	\ints\left(\uo-\frac{\sio}{\lambda_1}\right)^2+\ints\ut^2+\ints\vo^2+\ints\left(\vt-\frac{\sio a_2}{\lambda_1b_2}\right)^2\leq C_{19} \mathscr{E}_2(t)\leq \frac{C_{20}}{\epsilon_2 t+ C_{17}}.
\end{align*}
Finally, we obtain
{\small
	\begin{align*}
		\Big\|\uot-\frac{\sio}{\lambda_1}\Big\|_{\lis}+\Big\|\utt\Big\|_{\lis}+\Big\|\vot\Big\|_{\lis}+\Big\|\vtt-\frac{\sio a_2}{\lambda_1b_2}\Big\|_{\lis}\leq \frac{C_{22}}{\epsilon_2 t+ C_{17}},
	\end{align*}
} 
for all $t>t_0$, hence complete the proof.
\hfill\qedsymbol\\

\begin{remark}
	It is not generally possible to achieve exponential convergence by assuming $\eta_2=\frac{\sit \lambda_1}{\sio}$ in \eqref{l9.4}.
\end{remark}

\section{Bifurcation analysis}
We investigate cross-diffusion–driven instability of the spatially homogeneous steady state via standard linear stability analysis. Turing instability arises when a steady state that is locally asymptotically stable in the absence of diffusion becomes unstable under spatially heterogeneous perturbations due to cross-diffusion effects. To this end, we first examine the local stability of system \eqref{1.1} without diffusion. The corresponding non-diffusive system is given by
\begin{align}
	\left\{
	\begin{array}{llll}
	&\frac{\text{d} u_{1}}{\text{d}t}=\uo(\sio-\lambda_1\uo+\eta_1\ut),\\
	&\frac{\text{d} u_{2}}{\text{d}t}=\ut(\sit-\lambda_2\ut-\eta_2\uo),\\
	&\frac{\text{d} v_{1}}{\text{d}t}=\alo\ut-\beo\vo,\\
	&\frac{\text{d} v_{2}}{\text{d}t}=\alt \uo-\bet\vt,
	\end{array}
	\right.
\end{align}
Then the Jacobian matrix for coexistence steady state is given by
\begin{align}
	J=&
	\begin{pmatrix}
		\sio-2\lambda_1\uost+\eta_1\utst & \eta_1\uost & 0 & 0\\
		-\eta_2\utst & \sit-2\lambda_2\utst-\eta_2\uost & 0 & 0\\
		0 & \alo & -\beo & 0\\
		\alt & 0 & 0 & -\bet
	\end{pmatrix}.
\end{align}
Thus, we get
\begin{align*}
	J=
	\begin{pmatrix}
		-\lambda_1\uost & \eta_1\uost & 0 & 0 \\
		-\eta_2\utst & -\lambda_2\utst & 0 & 0\\
		0 & \alo & -\beo & 0\\
		\alt & 0 & 0 & -\bet
	\end{pmatrix}
	:=
	\begin{pmatrix}
		J_{11} & J_{12} & 0 & 0\\
		J_{21} & J_{22} & 0 & 0\\
		0 & \alo & -\beo & 0\\
		\alt & 0 & 0 & -\beta
	\end{pmatrix}.
\end{align*}
Here we use $\sio-\lambda_1\uost+\eta_1\utst=0$ and $\sit-\lambda_2\utst-\eta_2\uost=0$. Then $\text{Tr}(J)=-\lambda_1\uost-\lambda_2\utst-b_1-b_2<0$ and $\text{Det}(J)=\lambda_1\lambda_2\uost\utst b_1b_2+\eta_1\eta_2\uost\utst b_1b_2>0$. Therefore, the coexistence steady state $(\uost, \utst, \vost, \vtst)$ is locally asymptotically stable (LAS). 

\begin{figure}[H]\label{fig.5.1}
	\centering
	\subfloat[LAS of the ODE system]{\includegraphics[width=0.4\textwidth]{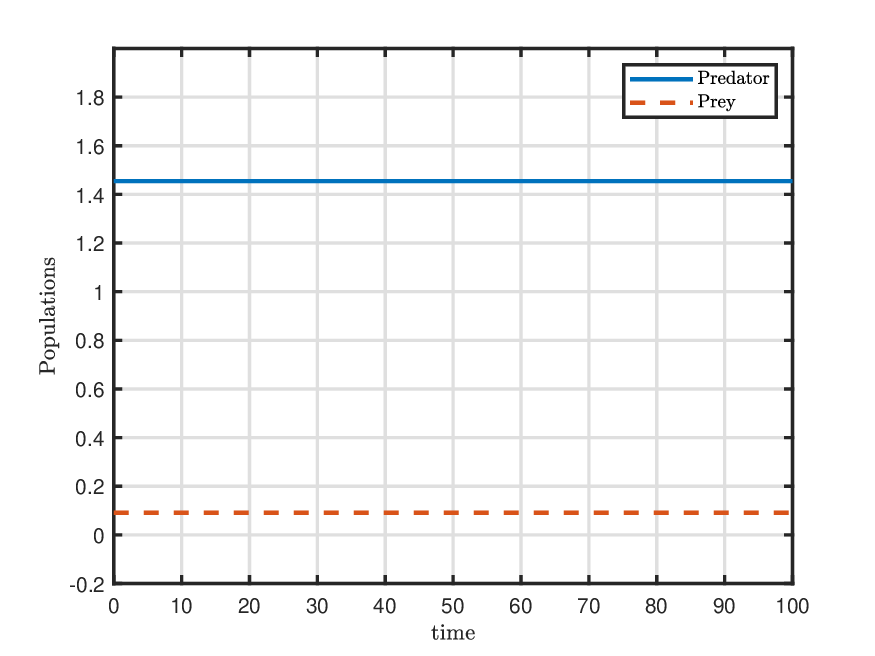}}
	\subfloat[The nullclines of predator and prey]{\includegraphics[width=0.4\textwidth]{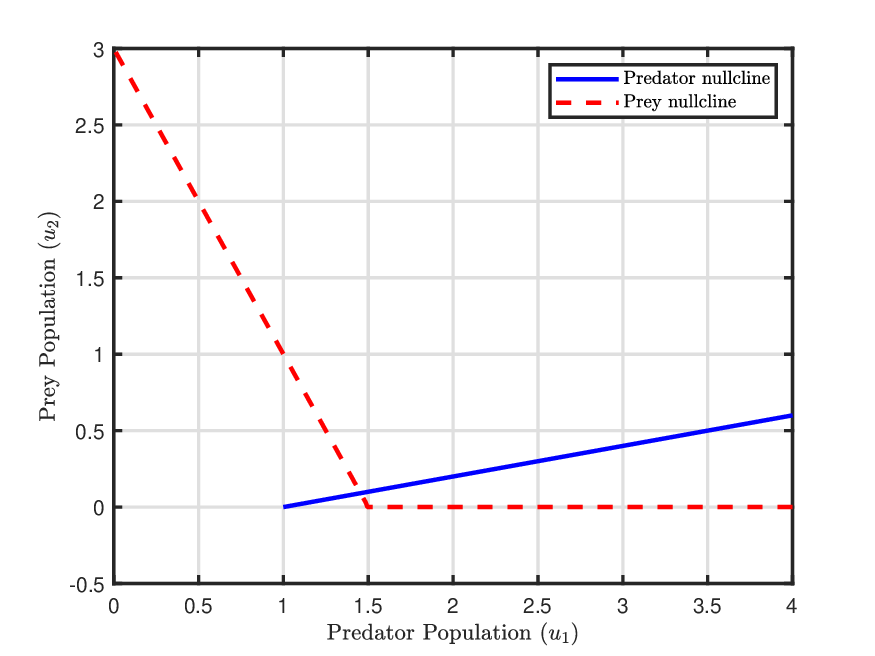}}
	\caption{Parameters $\sio=2, \sit=3, \lambda_1=2, \lambda_2=1, a_1=a_2=b_1=b_2=0.5, \eta_1=10, \eta_2=2$.}
\end{figure}
Turing patterns arise when diffusion destabilizes a spatially homogeneous steady state, leading to the spontaneous formation of stationary structures such as spots, stripes, or patches. In reaction--diffusion systems, this diffusion-driven instability amplifies specific spatial modes, producing ordered patterns without external forcing. In predator--prey models, Turing patterns explain how spatial segregation or aggregation can emerge purely from species interactions. Cross-diffusion further enhances this mechanism by widening the instability regime and enabling richer, biologically realistic pattern formation \cite{joydev1996}.
\begin{figure}[H]
	\subfloat[$u_1$ vs $ u_2$]{\includegraphics[width=0.33\textwidth]{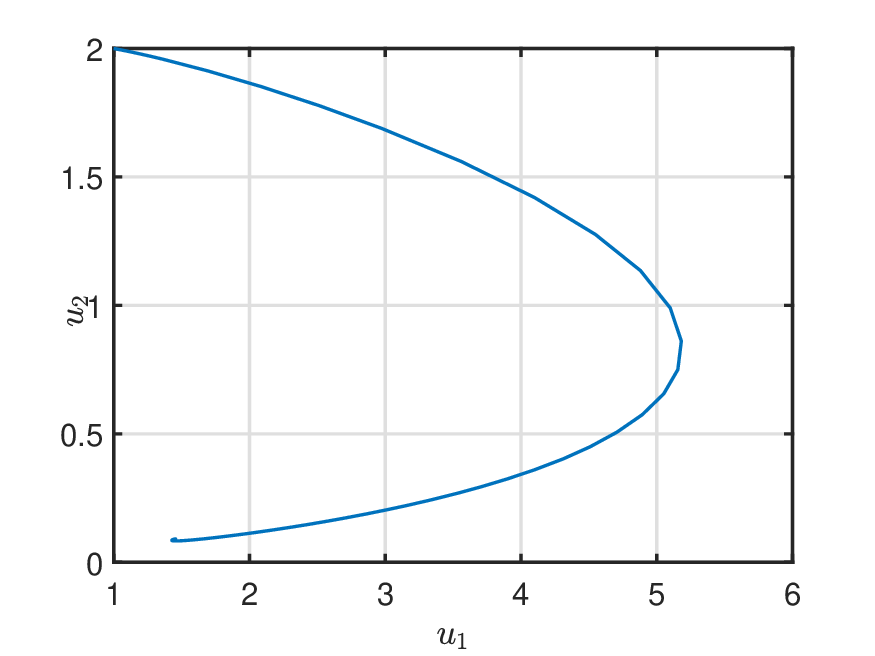}}
	\subfloat[$u_1$ vs $v_1$]{\includegraphics[width=0.33\textwidth]{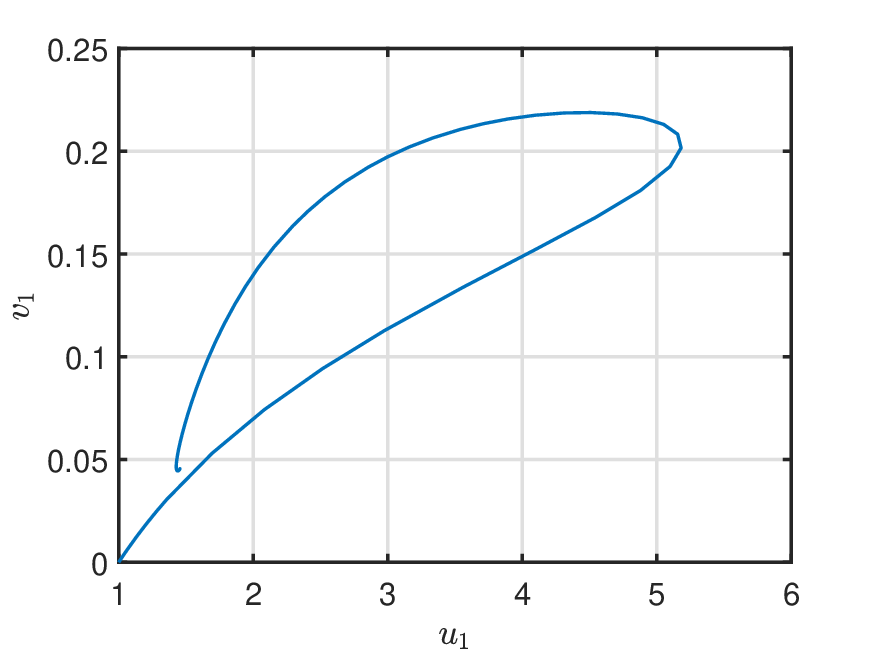}}
	\subfloat[$u_1$ vs $v_2$]{\includegraphics[width=0.33\textwidth]{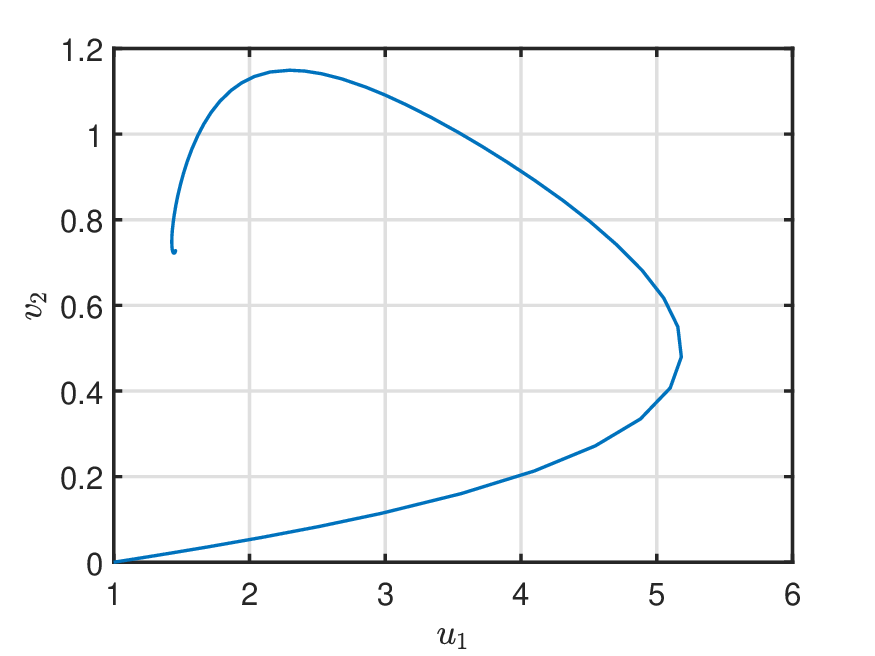}}\\
	\subfloat[$u_2$ vs $v_1$]{\includegraphics[width=0.33\textwidth]{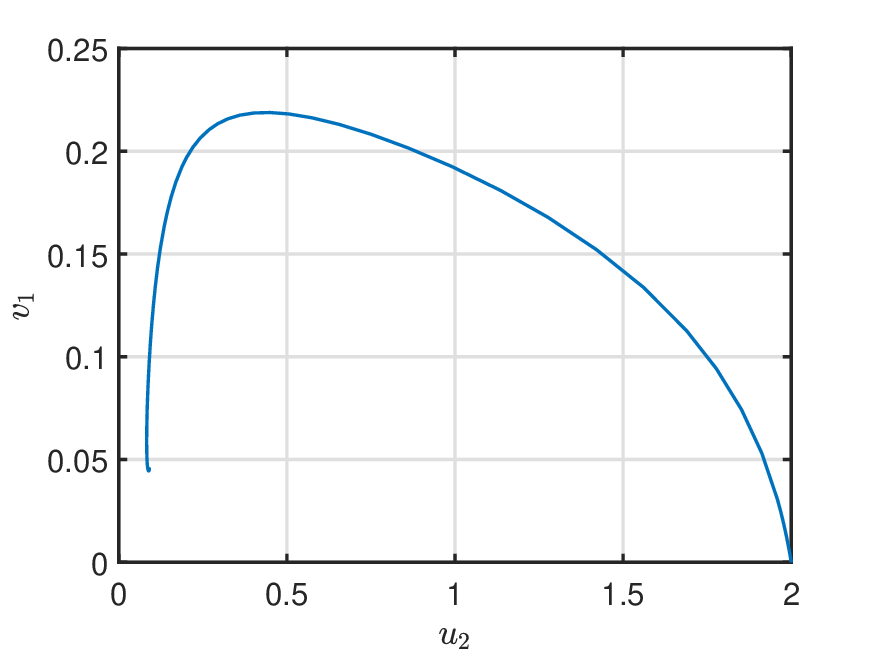}}
	\subfloat[$u_2$ vs $v_2$]{\includegraphics[width=0.33\textwidth]{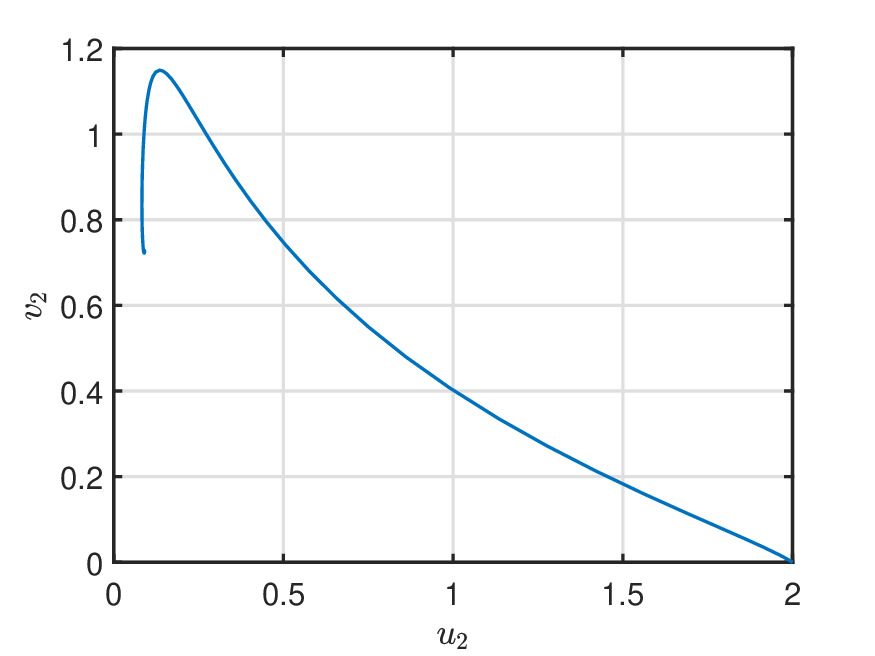}}
	\subfloat[$v_1$ vs $v_2$]{\includegraphics[width=0.33\textwidth]{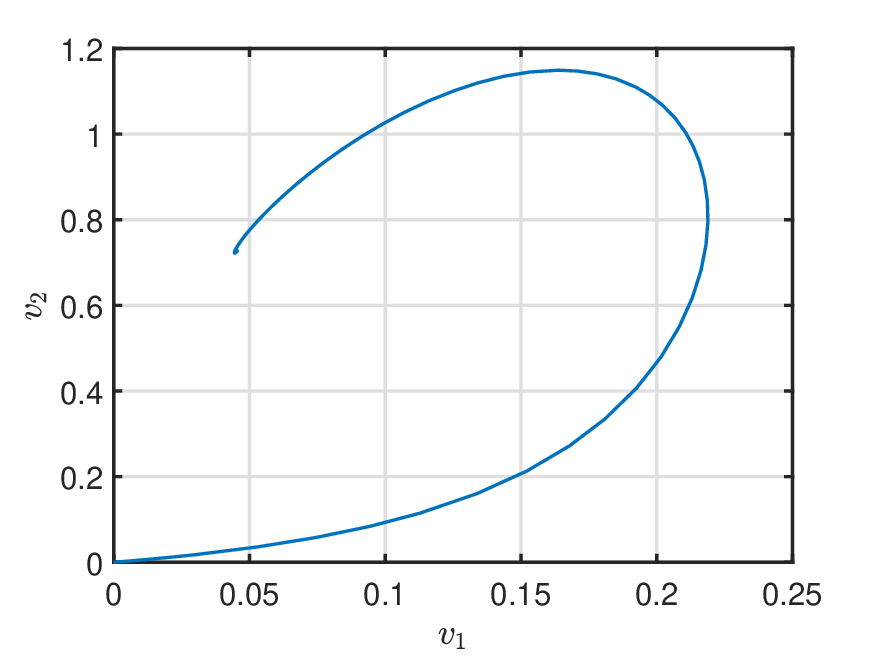}}
	\caption{Phase portrait of the species and chemicals. Parameter values $\sio=2, \sit=3, \lambda_1=2, \lambda_2=1, a_1=a_2=1, b_1=b_2=2, \eta_1=10, \eta_2=2$. The initial condition is $u_{10}=1, u_{20}=2, v_{10}=v_{20}=0$.}
\end{figure}
The above figure presents the phase portraits of the predator--prey system and their associated chemical signals for the given parameter set, where the trajectories represent the temporal evolution of the interacting biological populations and their chemical mediators. The $u_1$--$u_2$ plot illustrates the nonlinear interaction between predator and prey populations, indicating bounded coexistence dynamics. The portraits $u_1$--$v_1$ and $u_2$--$v_1$ describe the dependence of the prey-released chemical on prey density and its indirect influence on the predator population. Similarly, $u_1$--$v_2$ and $u_2$--$v_2$ show the relationship between predator density and the predator-associated chemical signal. The $v_1$--$v_2$ phase plot reflects the indirect coupling between the two chemical mediators through the predator--prey dynamics.

We now analyze the local stability of the system \eqref{1.1} with self-diffusion and assume that \eqref{1.3} hold. We consider small perturbations around the equilibrium
\begin{align*}
	u(x,t)=u^*+w(x,t),
\end{align*}
where $w=(\uo, \ut, \vo, \vt)^T$. Substituting into the system and linearizing, we obtain
\begin{align*}
	w_t=d\Delta w+Jw,
\end{align*}
where, $d$ is the diffusion matrix, $J$ is the Jacobian evaluated at $u^*$. We assume perturbations are of the form
\begin{align*}
	w(x,t)=ce^{\lambda t+ik\cdot x},
\end{align*}
where $k$ is the wave vector  (spatial frequency), $c$ is the time-dependent amplitude of the mode, $e^{\lambda t+ik\cdot x}$ is the sinusoidal spatial variation and $\lambda$ is the growth rate. Substituting, we get
\begin{align*}
	\lambda ce^{\lambda t+ik\cdot x}=&-k^2dce^{\lambda t+ik\cdot x}+Jce^{\lambda t+ik\cdot x}.
\end{align*}
The dispersion relation is given by $\text{Det}(-k^2d+J-\lambda I)=0$. That is
\begin{align*}
	\small
	\begin{vmatrix}
		-k^2d_{11}-\lambda_1\uost-\lambda & \eta_1\uost \\
		-\eta_2\utst & -k^2d_{21}-\lambda_2\utst-\lambda
	\end{vmatrix}
	=0.
\end{align*}
We can easily check Trace=$-k^2d_{11}-\lambda_1\uost-k^2d_{21}-\lambda_2\utst<0$ and Det=$(k^2d_{11}+\lambda_1\uost)(k^2d_{21}+\lambda_2\utst)+\eta_1\eta_2\uost\utst>0$. 

\begin{figure}[H]\label{fig.5.2}
	\centering
	\subfloat[Predator $(u_1)$]{\includegraphics[width=0.4\textwidth]{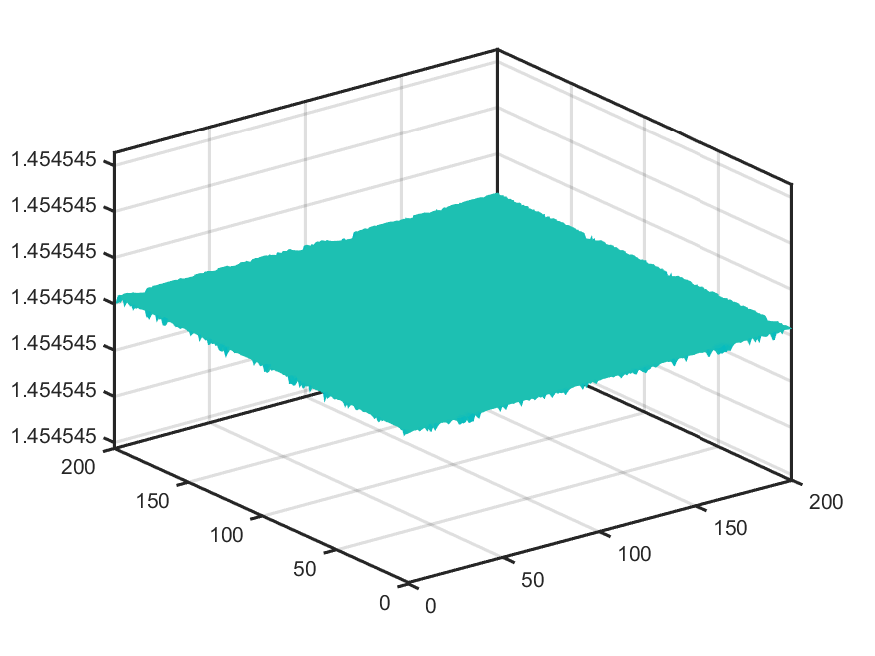}}
	\subfloat[Prey ($u_2$)]{\includegraphics[width=0.4\textwidth]{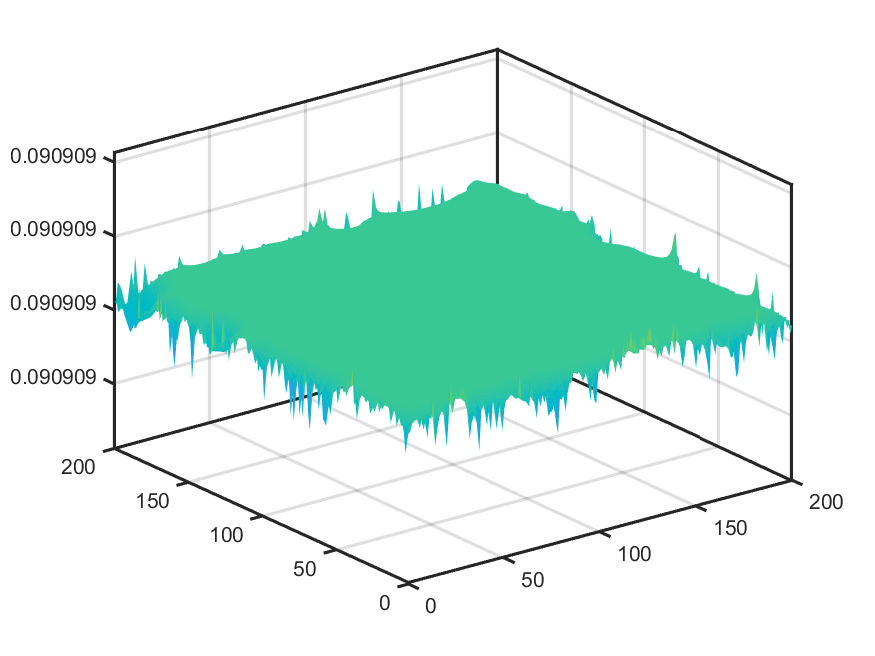}}
	\caption{LAS of the coexistence state of system \eqref{1.1} for the parameter values   $d_{11}=0.1, d_{12}=0, d_3=3, d_{21}=1, d_{22}=0, d_4=2, \sio=2, \sit=3, \lambda_1=2, \lambda_2=1, a_1=a_2=b_1=b_2=0.5, \eta_1=10, \eta_2=2$ and $t=1000$.}
\end{figure}
Therefore, the coexistence steady state 
$(\uost,\utst,\vost,\vtst)$ is locally asymptotically stable. Moreover, both the corresponding ODE system and the reaction–diffusion system incorporating self-diffusion remain locally asymptotically stable around this equilibrium point. Our analysis indicates that self-diffusion alone is insufficient to induce spatial pattern formation. Consequently, to explore the emergence of spatial patterns, we extend the model by incorporating cross-diffusion effects into the system. 
For the cross-diffusion system, we have
\begin{align*}
	\small
	\begin{vmatrix}
		-k^2d_{11}+J_{11}-\lambda & J_{12} & -k^2d_{12} & 0\\
		J_{21} & -k^2d_{21}+J_{22}-\lambda & 0 & k^2d_{22}\\
		0 & \alo & -k^2d_{3}-\beo-\lambda & 0\\
		\alt & 0 & 0 & -k^2d_{4}-\bet-\lambda
	\end{vmatrix}
	=0.
\end{align*}
General form of the characteristic polynomial is 
\begin{align}
	\lambda^4+A_1\lambda^3+A_2\lambda^2+A_3\lambda+A_4=0.\label{2}
\end{align}
Sufficient condition for cross-diffusion driven instability is $A_4<0$, where
\begin{align*}
	A_4=&k^8d_3 d_4 d_{11} d_{21}+k^6 \Big(\lambda_1 d_3 d_4 d_{21} u_1+\lambda_2 d_3 d_4 d_{11} u_2+b_1 d_4 d_{11} d_{21}+b_2 d_3 d_{11} d_{21}\Big)\\
	&+k^4 \Big(\lambda_2 b_1 d_4 d_{11} u_2+\lambda_2 b_2 d_3 d_{11} u_2+\lambda_1 b_1 d_4 d_{21} u_1+\lambda_1 b_2 d_3 d_{21} u_1+\lambda_1 \lambda_2 d_3 d_4 u_1 u_2\\
	&\qquad+\alpha _1 \alpha _2 d_{12} d_{22}+b_1 b_2 d_{11} d_{21}+d_3 d_4 \eta_1 \eta_2 u_1 u_2-a_1 d_4 d_{12} \eta_2 u_2-a_2 d_3 d_{22} \eta_1 u_1\Big)\\
	&+k^2 \Big(\lambda_1 \lambda_2 b_1 d_4 u_1 u_2+\lambda_1 b_2 b_1 d_{21} u_1+\lambda_2 b_2 b_1 d_{11} u_2+\lambda_1 \lambda_2 b_2 d_3 u_1 u_2+b_1 d_4 \eta_1 \eta_2 u_1 u_2\\
	&\qquad +b_2 d_3 \eta_1 \eta_2 u_1 u_2-a_2 b_1 d_{22} \eta_1 u_1-a_1 b_2 d_{12} \eta_2 u_2\Big)\\
	&+\lambda_1 \lambda_2 b_1 b_2 u_1 u_2+b_1 b_2 \eta_1 \eta_2 u_1 u_2
\end{align*}
This can be written as $h(k^2)=\omega_1k^4+\omega_2k^3+\omega_3k^2+\omega_4k+\omega_5$, where we assume $\omega_i>0, i=1,2,\cdots,5$.
Clearly,  $\omega_5>0$, this means the polynomial evaluated at $k^2=0$ is positive. The roots are given by
	\begin{align*}
		k_{1,2}^2=&-\frac{\omega_2}{4\omega_1}-\frac{1}{2}\beta_4
		\pm\frac{1}{2} \sqrt{\frac{\omega_2^2}{2 \omega_1^2}-\frac{4 \omega_3}{3 \omega_1}-\frac{2^\frac{1}{3}\beta_1}{3\omega_1\beta_3^\frac{1}{3}}-\frac{\beta_3^\frac{1}{3}}{3\times 2^\frac{1}{3}\omega_1}+\frac{\frac{\omega_2^3}{\omega_1^3}-\frac{4 \omega_3 \omega_2}{\omega_1^2}+\frac{8 \omega_4}{\omega_1}}{4\beta_4}},\\\\
		k_{3,4}^2=&-\frac{\omega_2}{4\omega_1}+\frac{1}{2}\beta_4
		\pm\frac{1}{2} \sqrt{\frac{\omega_2^2}{2 \omega_1^2}-\frac{4 \omega_3}{3 \omega_1}-\frac{2^\frac{1}{3}\beta_1}{3\omega_1\beta_3^\frac{1}{3}}-\frac{\beta_3^\frac{1}{3}}{3\times 2^\frac{1}{3}\omega_1}+\frac{\frac{\omega_2^3}{\omega_1^3}-\frac{4 \omega_3 \omega_2}{\omega_1^2}+\frac{8 \omega_4}{\omega_1}}{4\beta_4}},
	\end{align*}
	where
	\begin{align*}
		\beta_1=&\omega_3^2-3 \omega_2 \omega_4+12 \omega_1 \omega_5, \qquad 
		\beta_2=2 \omega_3^3-9 \omega_2 \omega_4 \omega_3-72 \omega_1 \omega_5 \omega_3+27 \omega_1 \omega_4^2+27 \omega_2^2 \omega_5,\\
		\beta_3=&\beta_2+\sqrt{-4\beta_1^3+\beta_2^2}, \qquad
		\beta_4=\sqrt{\frac{\omega_2^2}{4 \omega_1^2}-\frac{2 \omega_3}{3 \omega_1}+\frac{2^\frac{1}{3}\beta_1}{3\omega_1\beta_3^\frac{1}{3}}+\frac{\beta_3^\frac{1}{3}}{3\times 2^\frac{1}{3}\omega_1}}.
	\end{align*}
In general, deriving explicit analytical conditions for the onset of a Turing bifurcation is highly challenging due to the complexity of the associated characteristic equations. Consequently, we rely on numerical simulations to identify the emergence of Turing patterns and to determine the corresponding admissible values of $k^{2}$. 

\begin{figure}[H]\label{fig.5.3}
	\centering
	\subfloat[Predator $(u_1)$]{\includegraphics[width=0.4\textwidth]{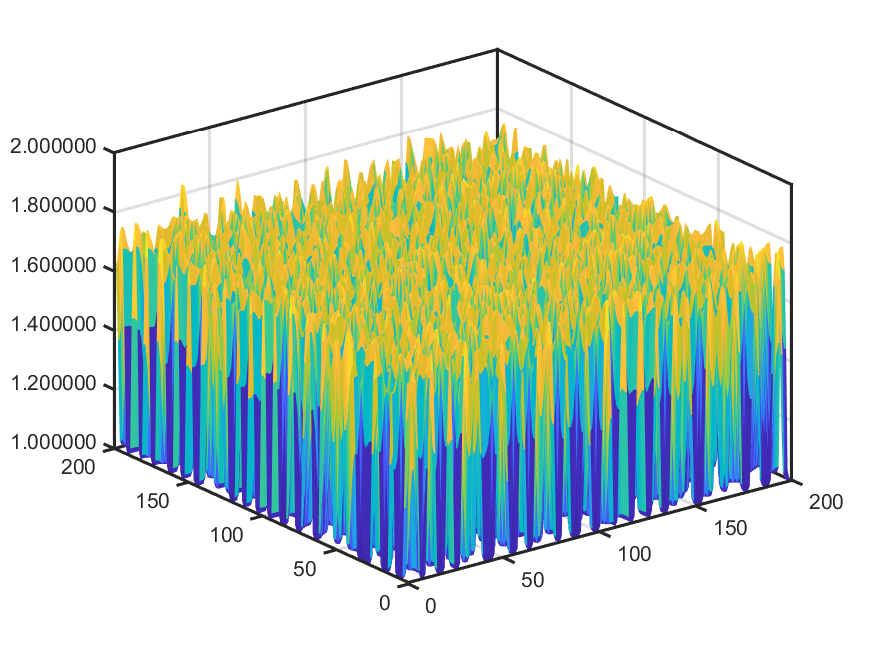}}
	\subfloat[Prey ($u_2$)]{\includegraphics[width=0.4\textwidth]{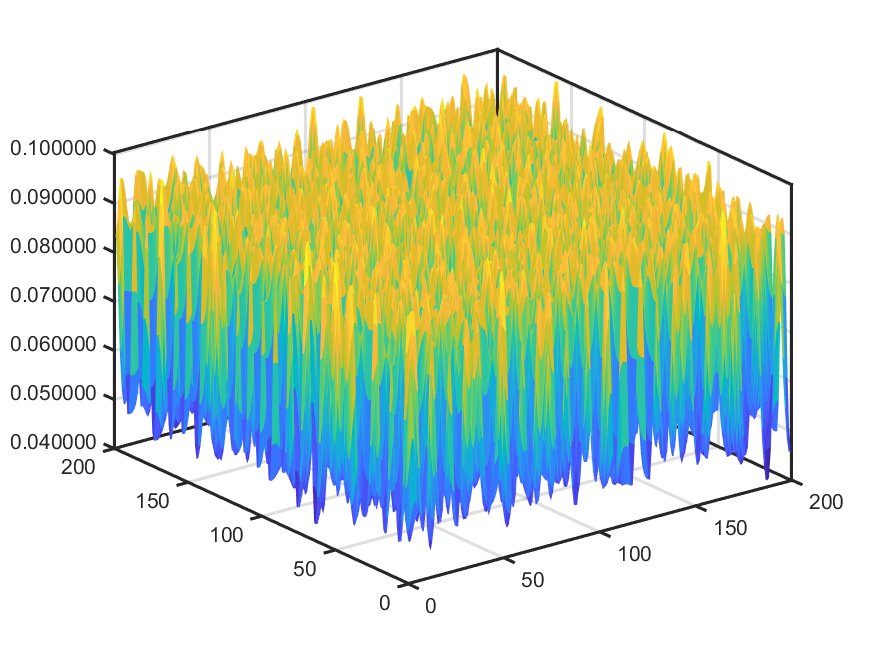}}
	\caption{Cross-diffusion driven instability of the system \eqref{1.1} for the parameter values   $d_{11}=0.1, d_{12}=1, d_3=3, d_{21}=1, d_{22}=2, d_4=2, \sio=2, \sit=3, \lambda_1=2, \lambda_2=1, a_1=a_2=b_1=b_2=0.5, \eta_1=10, \eta_2=2$ and $t=1000$.}
\end{figure}

We now introduce the finite difference method to conduct numerical simulations and investigate the temporal and spatial dynamics of the predator--prey system \eqref{1.1}. We consider the square domain $[0,200]^{2}$ with a maximum simulation time of $T_{\text{max}} = 5000$. To implement our iterative scheme, we choose spatial steps $\Delta x = \Delta y = 1$ and a time step $\Delta t = 0.01$. In addition, the initial population distributions are taken as small-amplitude random perturbations around the coexistence steady state $(u_{1}^{*}, u_{2}^{*}, v_{1}^{*}, v_{2}^{*})$ as
\begin{align*}
	\uo(\cdot, 0) &= \uost - 2 \times 10^{-4}\,\psi 
	\quad \text{and} \quad 
	\ut(\cdot, 0) = \utst + 2 \times 10^{-4}\,\psi,\\
	\vo(\cdot, 0) &= \vo^* + 2 \times 10^{-4}\,\psi \quad \text{and} \quad
	\vt(\cdot, 0) = \vt^* - 2 \times 10^{-4}\,\psi, 
\end{align*}
where $\psi$ is a random-number matrix whose entries lie between $0$ and $200$. Using the parameter sets employed in Figures~\ref{fig.5.5} - \ref{fig.5.10}, we determine the positive ranges of $k^{2}$ that give rise to diffusion-driven instability, as summarized in the following table.
\begin{table}[H]\label{tab.5.1}
	\centering
	\caption{Turing bifurcation is observed for $k^2$ in Figures \ref{fig.5.5} - \ref{fig.5.10}.}
	\begin{tabular}{c c c}
		\toprule
		{\bf Figure} & {\bf Parameter values } & {\bf range of $k^2$}\\
		\midrule
		\multirow{3}{*}{\ref{fig.5.5}} & $\eta_2=2$ & [0.253661, 0.925807]\\
		& $\eta_2=2.25$ & [0.138333, 1.255916]\\
		& $\eta_2=2.5$ & [0.074781, 1.526103]\\
		\midrule
		\multirow{3}{*}{\ref{fig.5.6}} & $\eta_2=2.65$ & [0.047290, 1.673559]\\
		& $\eta_2=2.75$ & [0.031782, 1.767353]\\
		& $\eta_2=2.9$ & [0.011695, 1.902321]\\
		\midrule
		\multirow{3}{*}{\ref{fig.5.7}} & $\eta_1=5$, $\eta_2=0.95$ & [0.026322, 0.092956]\\
		& $\eta_1=10$, $\eta_2=0.25$ & [0.100961, 0.512843]\\
		& $\eta_1=10$, $\eta_2=0.75$ & [0.027417, 0.442167]\\
		\midrule
		\multirow{2}{*}{\ref{fig.5.8}} & $\eta_1=5$, $\eta_2=0.075$ & [0.001004, 0.4337821]\\
		&  $\eta_1=10$, $\eta_2=0.075$ & [0.000455, 0.803716]\\
		\midrule
		\multirow{2}{*}{\ref{fig.5.9}} & $\eta_1=10$, $\eta_2=1$ & [0.185205, 0.581330]\\
		& $\eta_1=10$, $\eta_2=1.5$ & [0.110821, 0.675433]\\
		\midrule
		\multirow{2}{*}{\ref{fig.5.10}} & (a) & [1.102053, 12.303332]\\
		& (b) & [0.485184, 5.105713]\\
		\bottomrule
	\end{tabular}
\end{table}
Interestingly, predator ($u_1$) and prey ($u_2$) often exhibit maxima and minima at the same spatial locations because their movements are strongly regulated by the prey chemical ($v_1$) and predator chemical ($v_2$). These chemical cues act as shared environmental drivers, causing both species to aggregate in regions where $v_1$ enhances prey growth and simultaneously attracts predators. Likewise, areas with low chemical signal levels become unfavourable for both, leading to co-located minima. As a result, the dominant unstable mode in the Turing instability drives $u_1$, $u_2$, $v_1$, and $v_2$ to vary in phase, producing synchronized predator–prey hotspots rather than spatially offset distributions typical of classical predator–prey interactions.
\begin{figure}[H]\label{fig.5.4}
	\centering
	\subfloat[]{\includegraphics[width=0.45\textwidth]{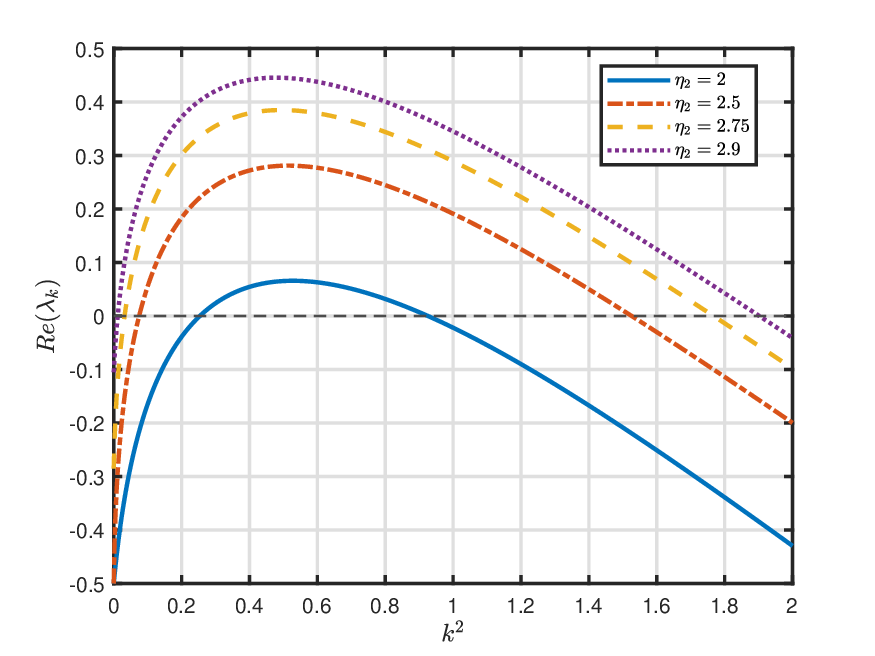}}
	\subfloat[]{\includegraphics[width=0.45\textwidth]{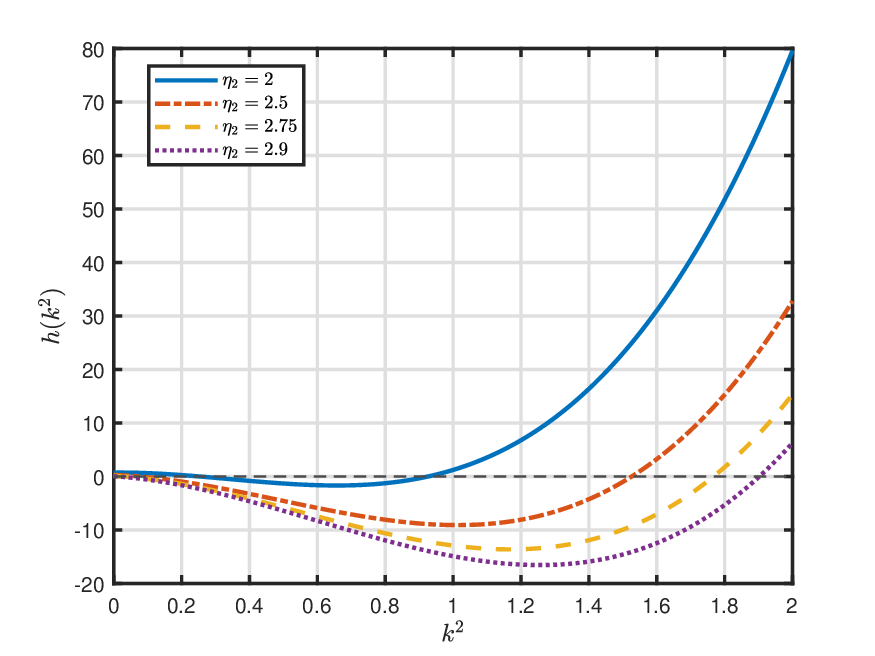}}
	\caption{(a) Dispersion relation and (b) $h(k^2)$ with respect to $k^2$, where $d_{11}=0.1, d_{12}=1, d_3=3, d_{21}=1, d_{22}=2, d_4=2, \sio=2, \sit=3, \lambda_1=2, \lambda_2=1, a_1=a_2=b_1=b_2=0.5, \eta_1=10$.}
\end{figure}

\begin{table}[H]
	\small
	\centering
	\caption{Comparative summary of Turing pattern formation under various parameter regimes.}
	\begin{tabular}{|>{\centering\arraybackslash}p{1.25cm}|>{\centering\arraybackslash}p{3.5cm}|>{\centering\arraybackslash}p{4.5cm}|p{5cm}|}
		\hline
		\textbf{Figure} & \textbf{Key Parameters Varied} & \textbf{Pattern Type \linebreak (Predator \& Prey)}  & \textbf{Observations} \\ 
		\hline
		
		\ref{fig.5.5} & $\eta_2 = 2, 2.25, 2.5$ & Spot-Stripes $\to$ labyrinthine $\to$ localized spot-stripes. & Increasing predation rate ($\eta_2$) reduces predator–prey spatial mixing. Predator aggregation becomes more localized, implying stronger prey avoidance and increased predator clustering around favourable patches. \\ 
		\hline
		
		\ref{fig.5.6} & $\eta_2 = 2.65, 2.75$, $2.9$ & localized spot-stripes \linebreak $\to$ hexagonal spots.  & Consistent with the first figure, we find that isolated coexistence spots emerge when $\eta_2\approx \frac{\sit \lambda_1}{\sio}$.\\ 
		\hline
		
		\ref{fig.5.7} & $\eta_1 = 5, \eta_2=0.95$ \linebreak and $\eta_1=10, \linebreak \eta_2 = 0.25, 0.75$
		& Labyrinths / stripes-spots / spots.  &  Here, a higher $\eta_1$ promotes the formation of isolated coexistence spots.\\ 
		\hline
		
		\ref{fig.5.8} & Small $\sio, \sit$, $\lambda_1, \lambda_2$ & Regular isolated hexagonal spot and dense square spot &  An increase in $\eta_1$ 	leads to a transition from isolated coexistence spots to dense, square-shaped patches\\ 
		\hline
		
		\ref{fig.5.9} & $\eta_2 = 1, 1.5$ & Labyrinths (lattice stripe) \linebreak $\to$ refined stripes & Balanced growth and diffusion yield stable coexistence bands \\ 
		\hline
		
		\ref{fig.5.10} & Coefficients $d_{i,j}$ & Disordered / irregular patterns  & Cross-diffusion amplifies segregation scale but preserves predator–prey synchrony \\ 
		\hline
	\end{tabular}
\end{table}

\begin{figure}[H]
	\centering
	\includegraphics[width=0.3\textwidth]{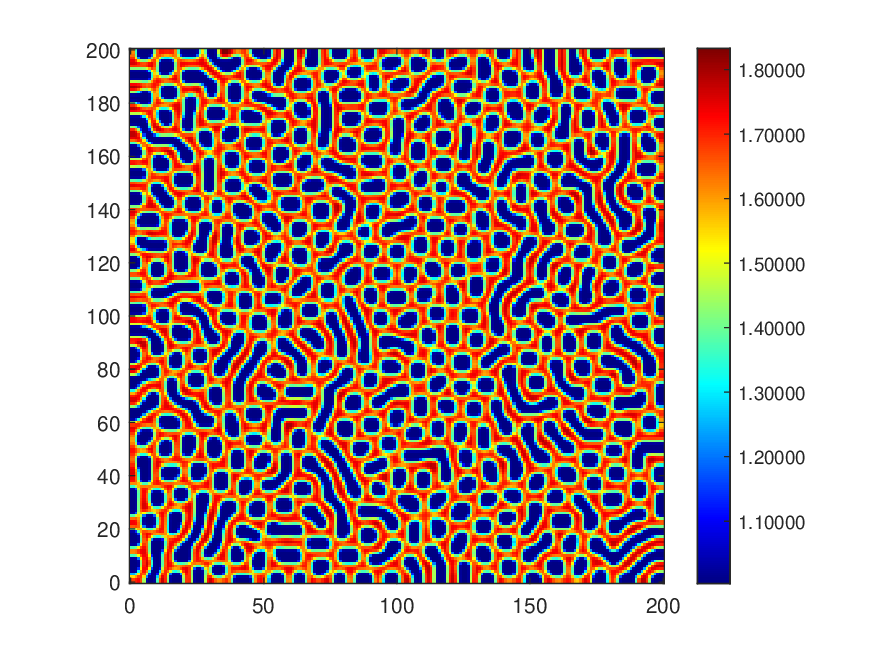}
	\includegraphics[width=0.3\textwidth]{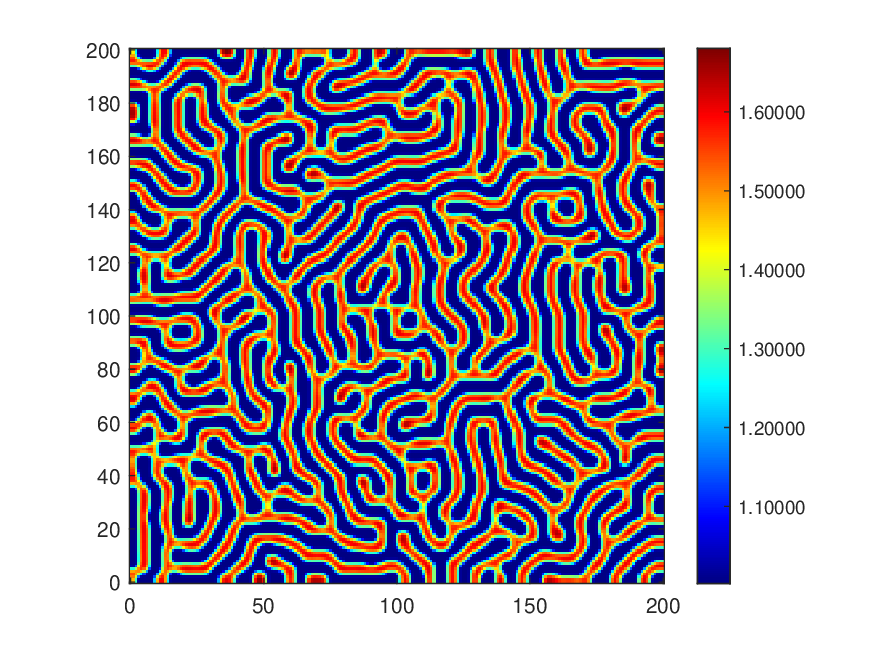}
	\includegraphics[width=0.3\textwidth]{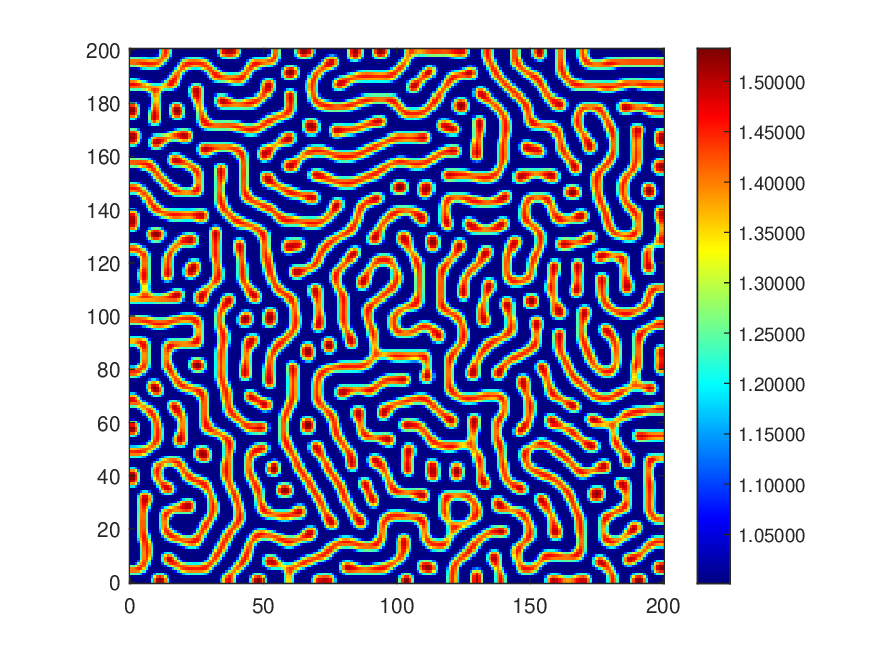}\\
	\subfloat[$\eta_2=2$]{\includegraphics[width=5cm, height=4cm]{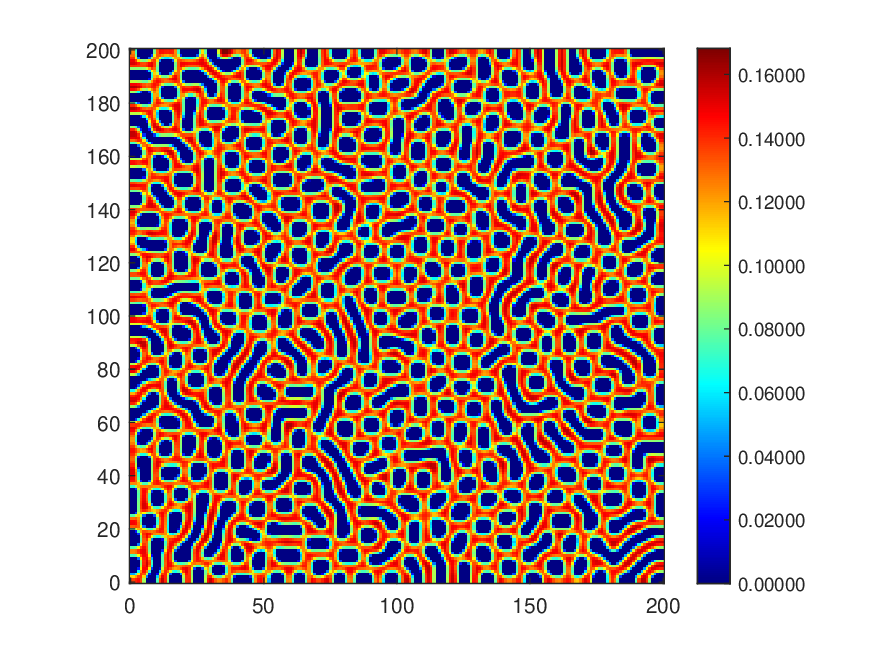}}
	\subfloat[$\eta_2=2.25$]{\includegraphics[width=5cm, height=4cm]{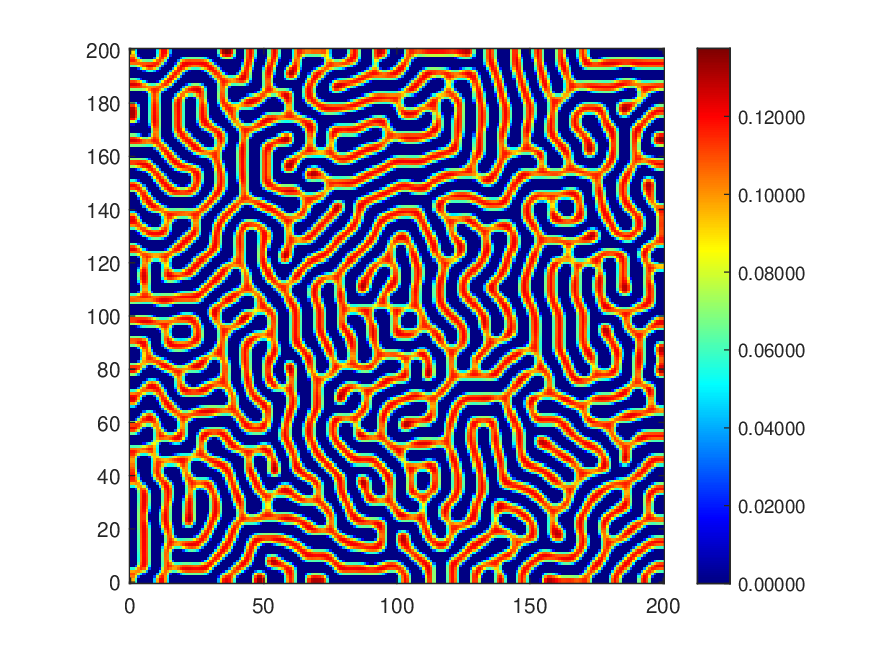}}
	\subfloat[$\eta_2=2.5$]{\includegraphics[width=5cm, height=4cm]{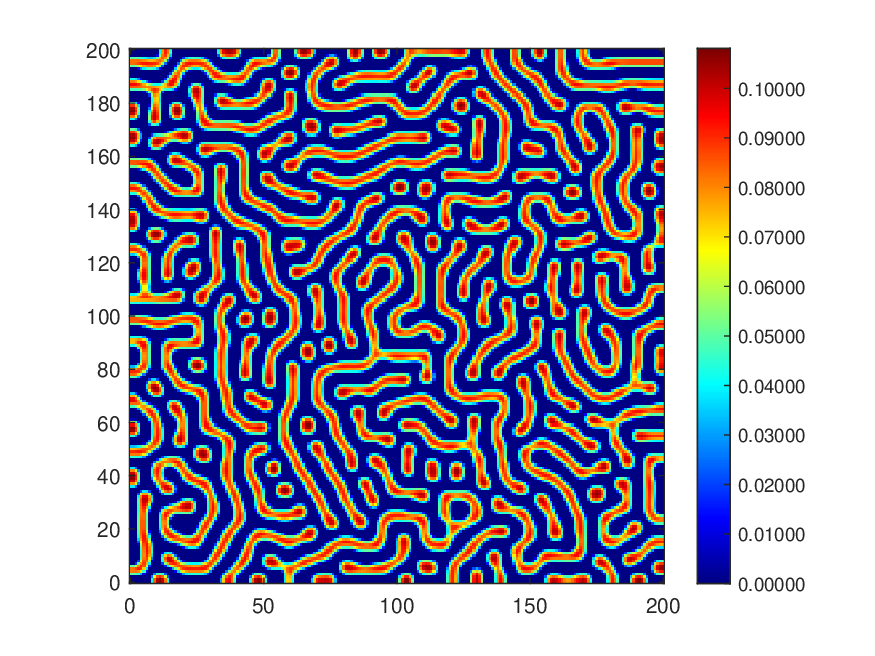}}
	\caption{Spatial distributions of the species Predator $u_1$ (first row) and Prey $u_2$ (second row), where $d_{11}=0.1, d_{12}=1, d_3=3, d_{21}=1, d_{22}=2, d_4=2, \sio=2, \sit=3, \lambda_1=2, \lambda_2=1, a_1=a_2=b_1=b_2=0.5, \eta_1=10$.}
	\label{fig.5.5}
\end{figure}

\begin{figure}[H]
	\centering
	\includegraphics[width=5cm,height=4cm]{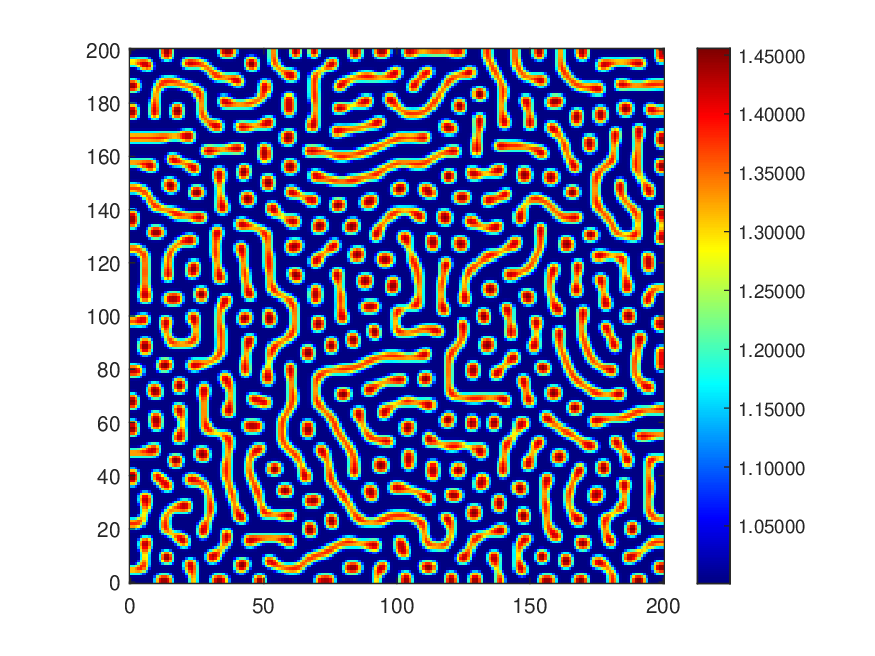}
	\includegraphics[width=5cm,height=4cm]{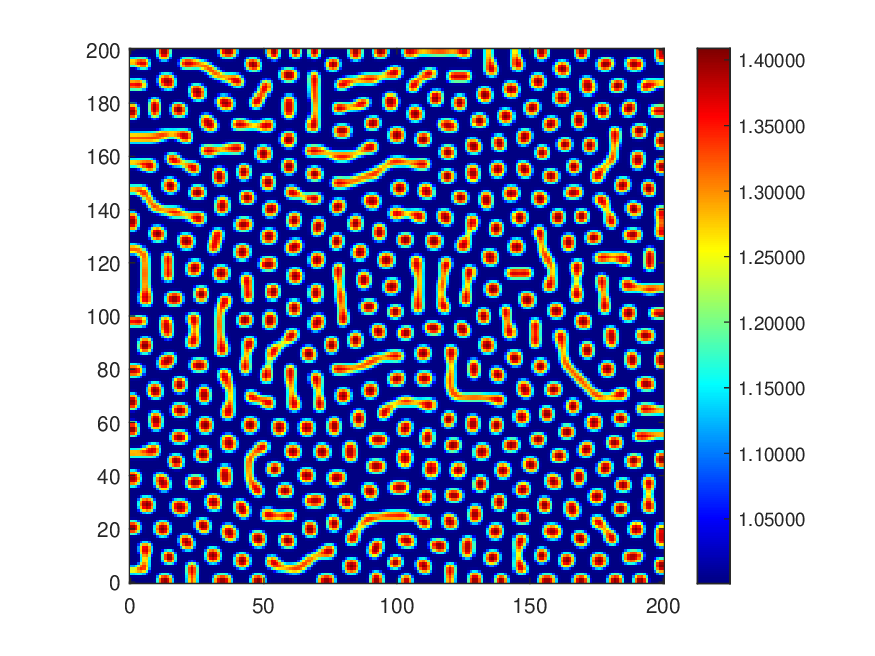}
	\includegraphics[width=5cm,height=4cm]{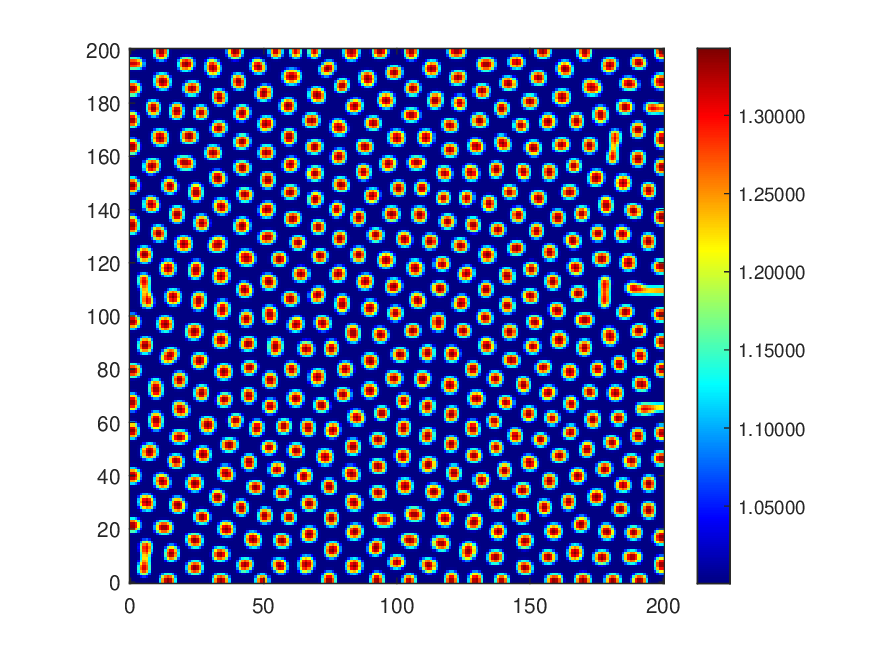}
	\subfloat[$\eta_2=2.65$]{\includegraphics[width=5cm, height=4cm]{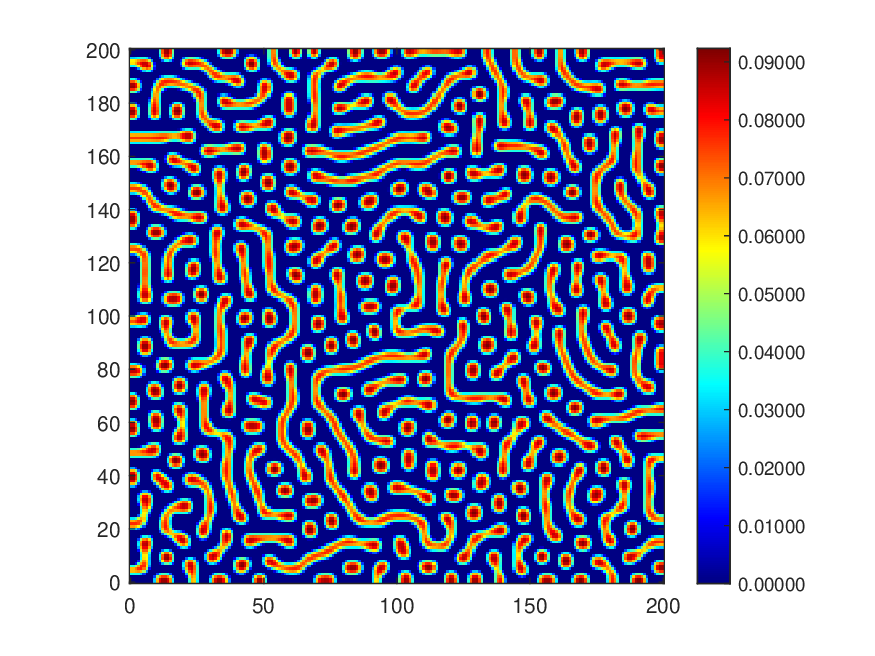}}
	\subfloat[$\eta_2=2.75$]{\includegraphics[width=5cm, height=4cm]{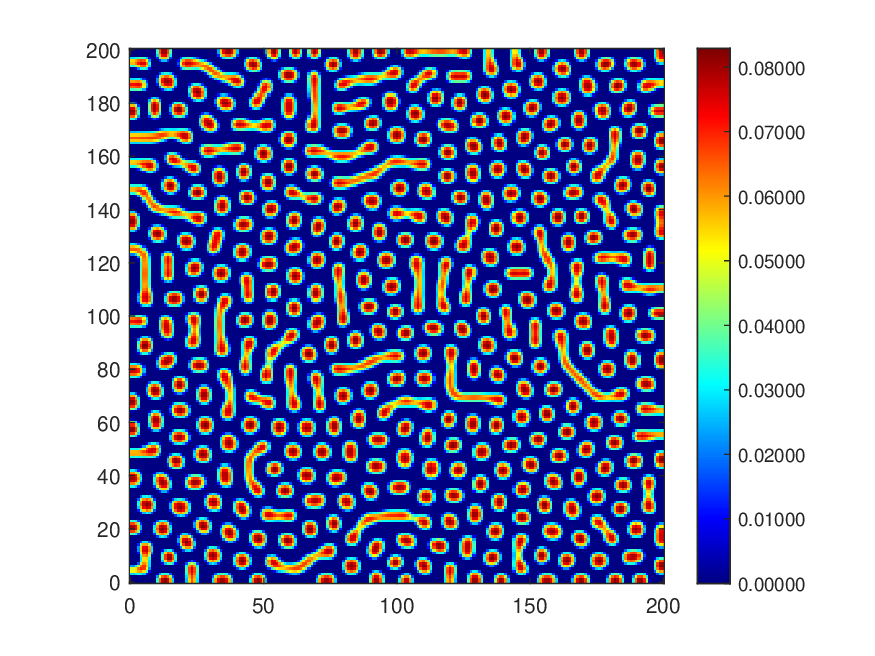}}
	\subfloat[$\eta_2=2.9$]{\includegraphics[width=5cm, height=4cm]{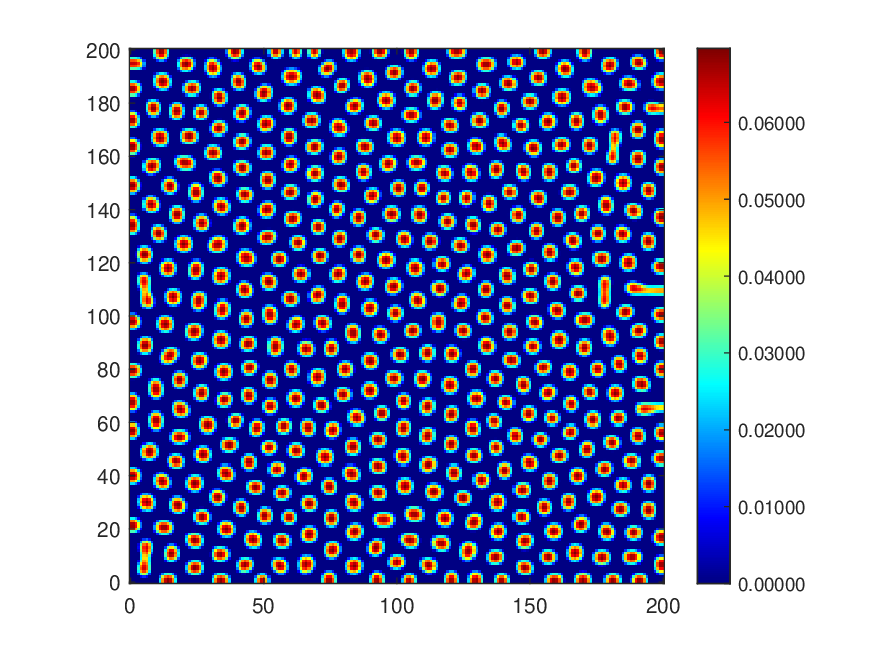}}
	\caption{Spatial distributions of the species Predator $u_1$ (first row) and Prey $u_2$ (second row), where $d_{11}=0.1, d_{12}=1, d_{21}=1, d_{22}=2, d_3=3, d_4=2, \sio=2, \sit=3, \lambda_1=2, \lambda_2=1, a_1=a_2=b_1=b_2=0.5, \eta_1=10$.}
	\label{fig.5.6}
\end{figure}

\begin{figure}[H]
	\centering
	\includegraphics[width=5cm, height=4cm]{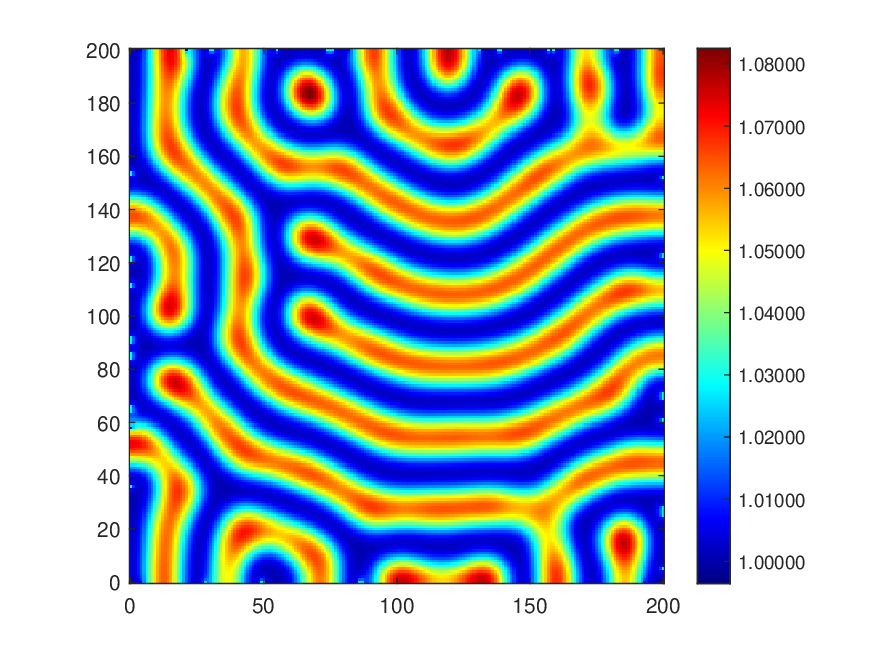}
	\includegraphics[width=5cm, height=4cm]{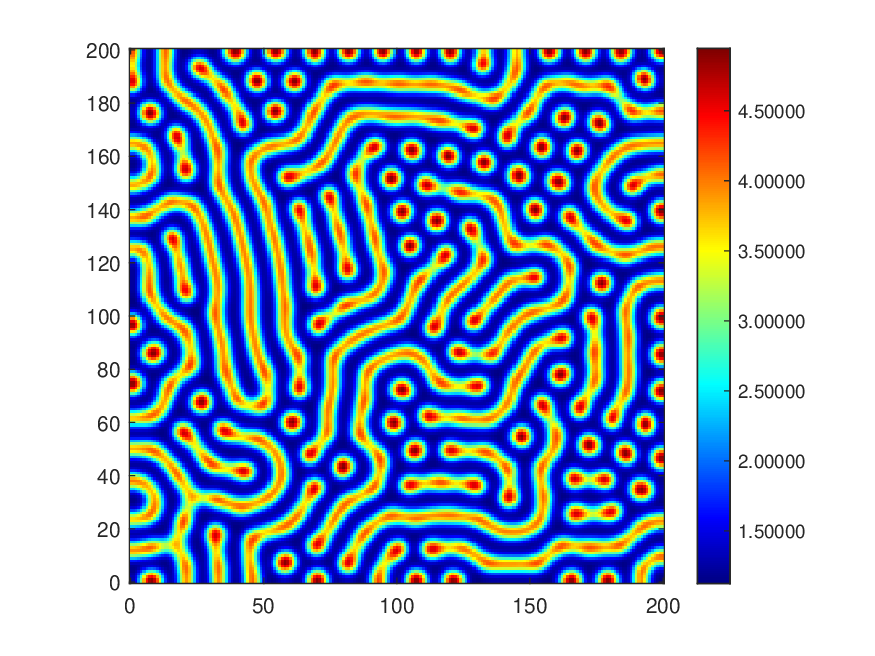}
	\includegraphics[width=5cm, height=4cm]{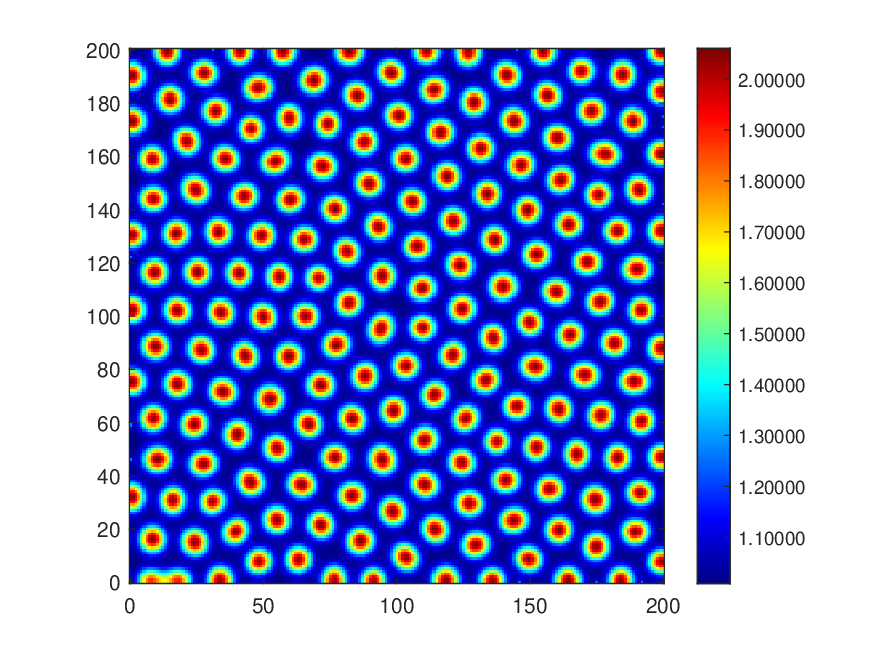}\\
	\subfloat[$\eta_1=5, \eta_2=0.95$]{\includegraphics[width=5cm, height=4cm]{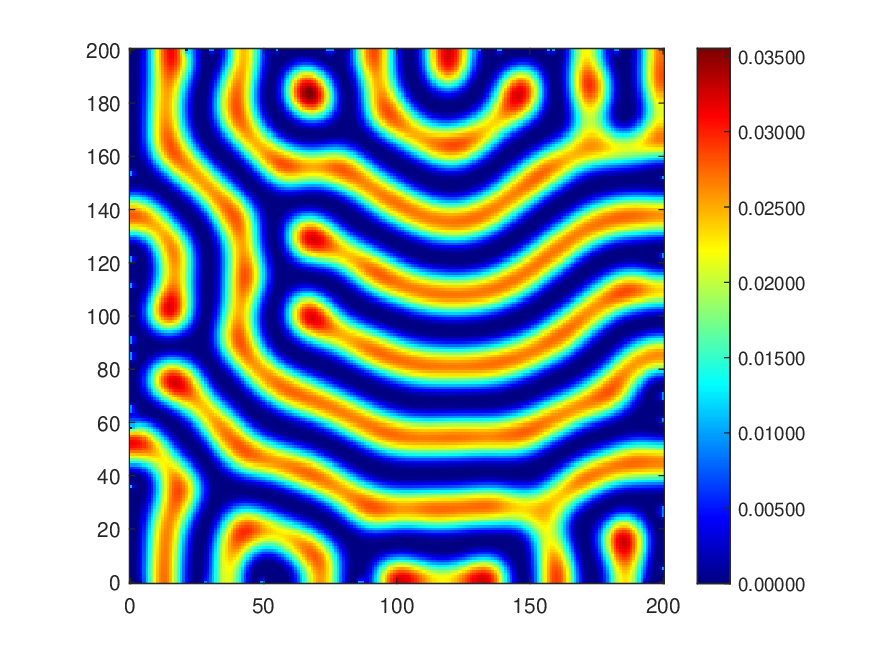}}
	\subfloat[$\eta_1=10, \eta_2=0.25$]{\includegraphics[width=5cm, height=4cm]{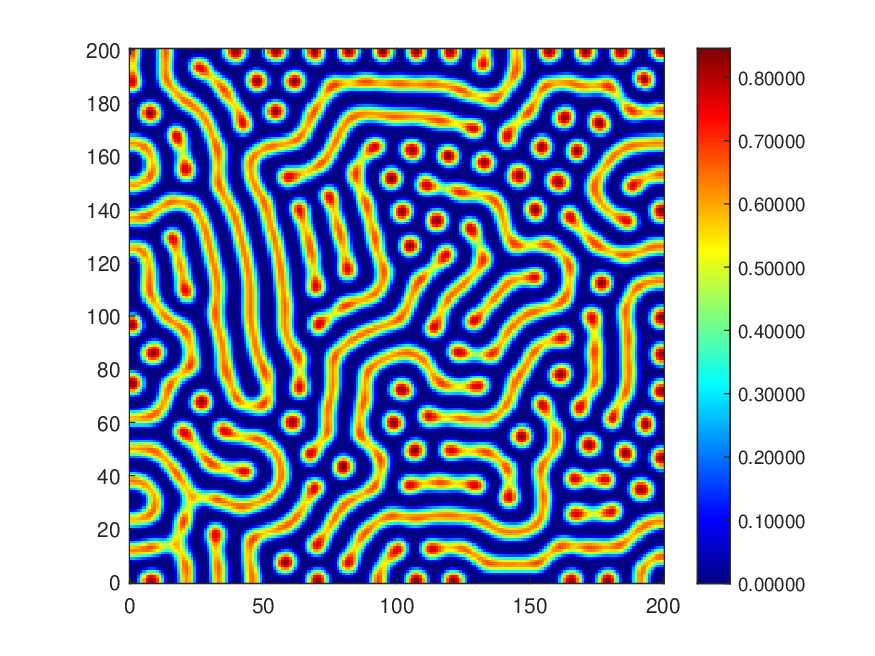}}
	\subfloat[$\eta_1=10, \eta_2=0.75$]{\includegraphics[width=5cm, height=4cm]{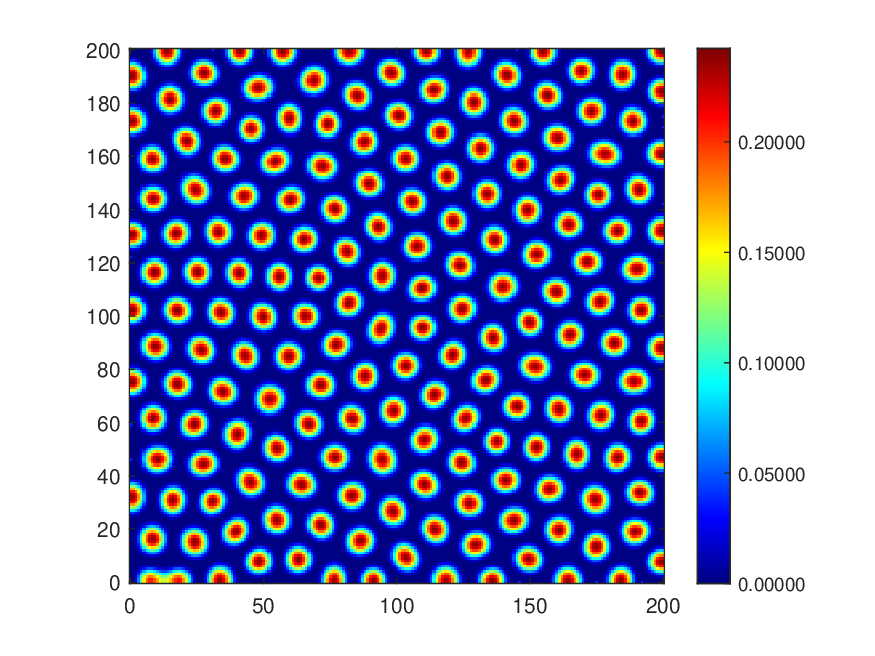}}
	\caption{Spatial distributions of the species Predator $u_1$ (first row) and Prey $u_2$ (second row), where $d_{11}=3.25, d_{12}=3, d_3=0.75, d_{21}=5, d_{22}=3, d_{4}=0.75, a_1=a_2=b_1=b_2=0.5, \sio=2, \sit=1, \lambda_1=2, \lambda_2=1$.}
	\label{fig.5.7}
\end{figure}

\begin{figure}[H]
	\centering
	\includegraphics[width=5.5cm, height=4cm]{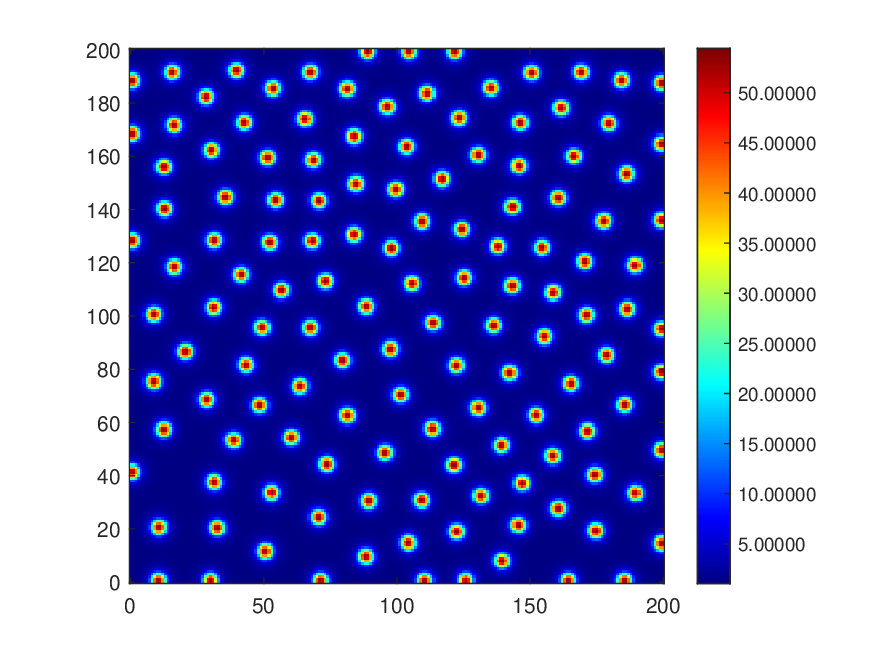}
	\includegraphics[width=5.5cm, height=4cm]{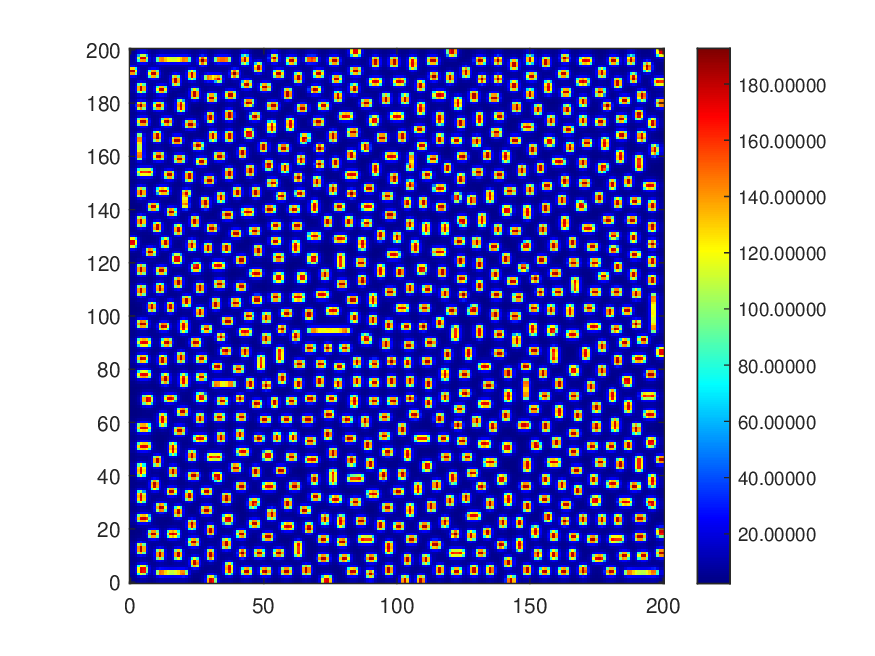}\\
	\subfloat[$\eta_1=5, \eta_2=0.075$]{\includegraphics[width=5.5cm, height=4cm]{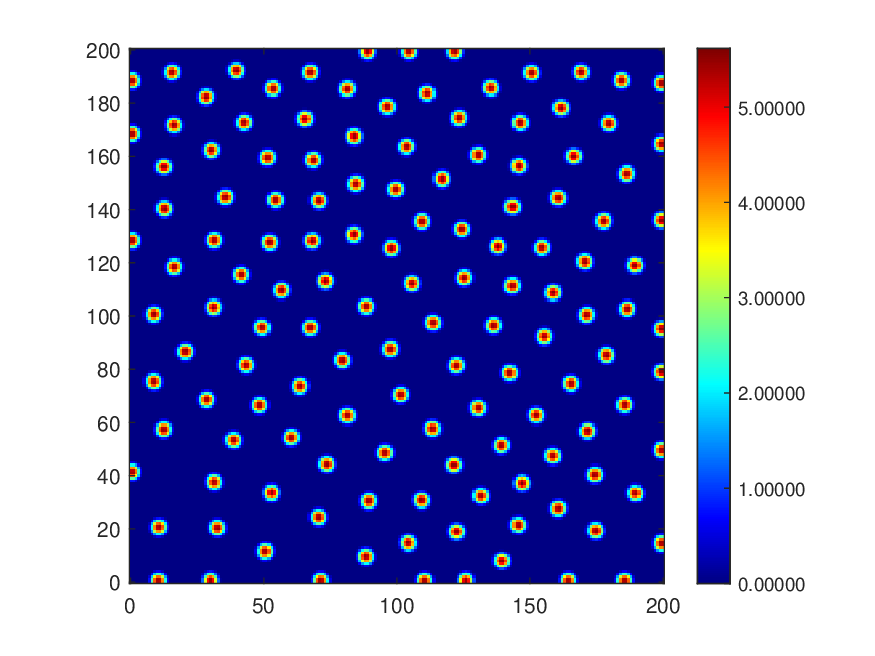}}
	\subfloat[$\eta_1=10, \eta_2=0.075$]{\includegraphics[width=5.5cm, height=4cm]{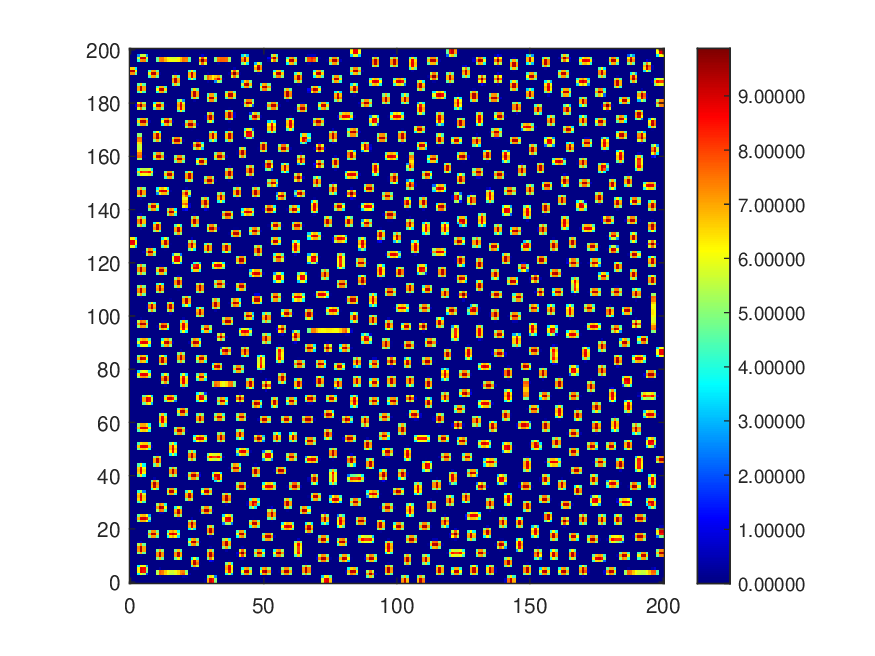}}
	\caption{Spatial distributions of the species Predator $u_1$ (first row) and Prey $u_2$ (second row), where $d_{11}=3.25, d_{12}=3, d_3=0.75, d_{21}=5, d_{22}=3, d_{4}=0.75, a_1=a_2=b_1=b_2=0.5, \sio=0.5, \sit=0.1, \lambda_1=0.5, \lambda_2=0.1$.}
	\label{fig.5.8}
\end{figure}

\begin{figure}[H]
	\centering
	\includegraphics[width=5.5cm, height=4cm]{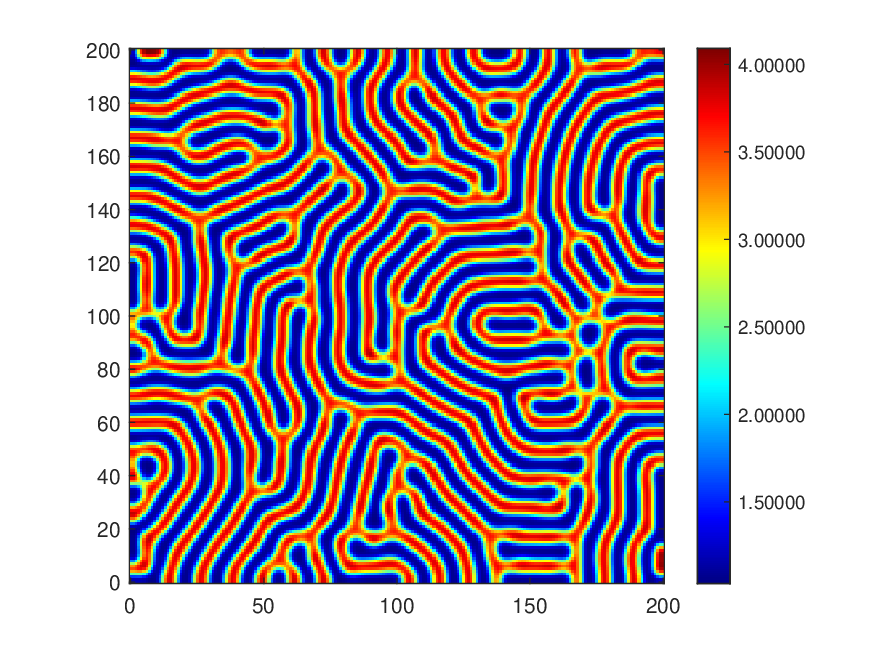}
	\includegraphics[width=5.5cm, height=4cm]{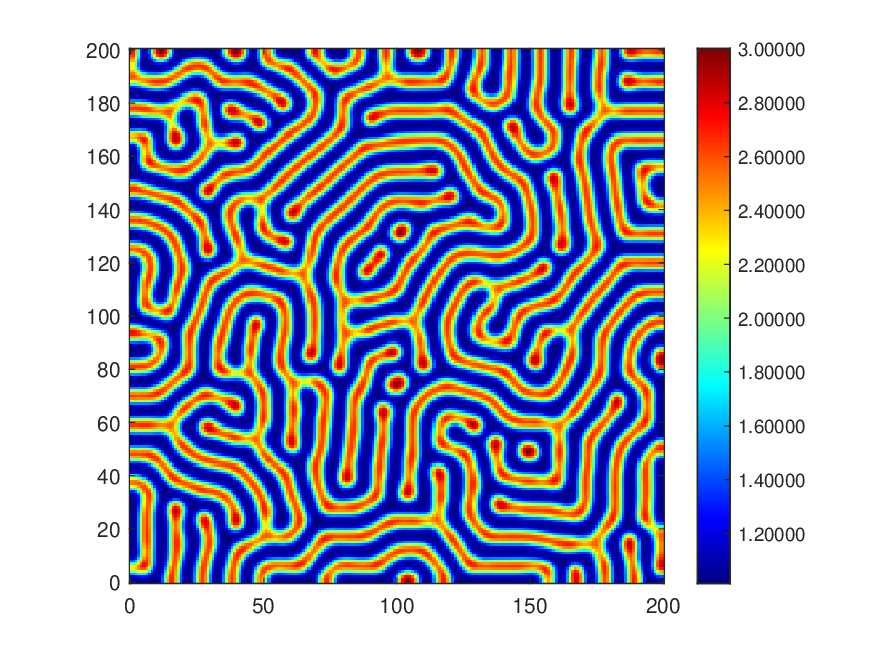}\\
	\subfloat[Prey $\eta_2=1$]{\includegraphics[width=5.5cm, height=4cm]{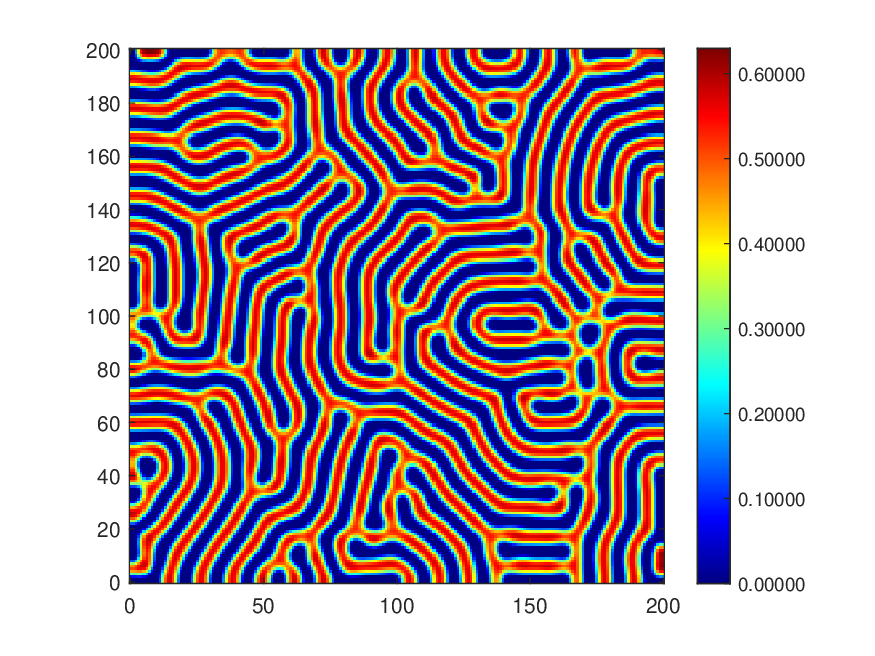}}
	\subfloat[Prey $\eta_2=1.5$]{\includegraphics[width=5.5cm, height=4cm]{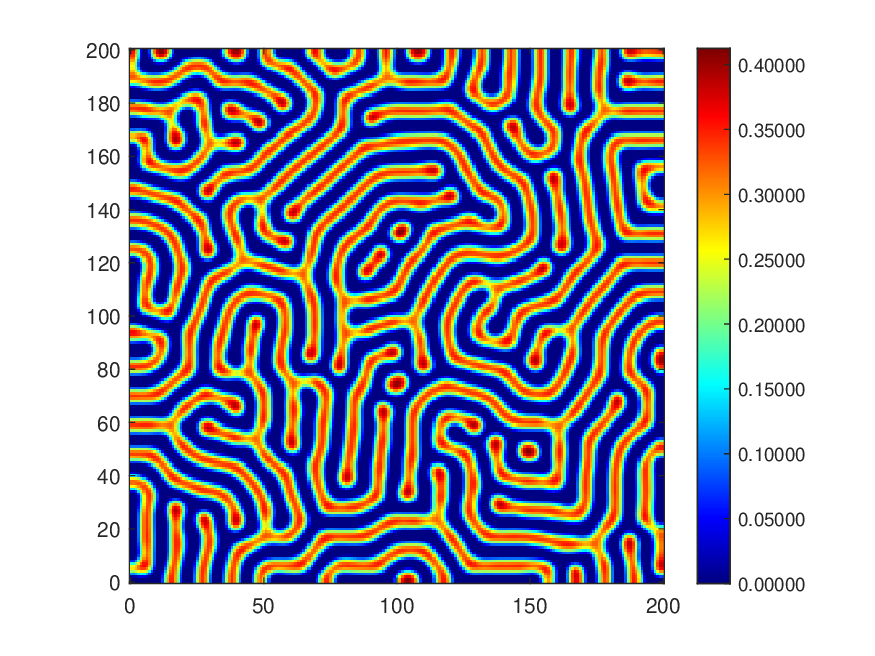}}
	\caption{Spatial distributions of the species Predator $u_1$ (first row) and Prey $u_2$ (second row), where $d_{11}=1, d_{12}=1, d_3=1, d_{21}=5, d_{22}=5, d_{4}=1.5, a_1=a_2=b_1=b_2=0.5, \sio=2, \sit=3, \lambda_1=2, \lambda_2=1, \eta_1=10$.}
	\label{fig.5.9}
\end{figure}

\begin{figure}[H]
\centering
	\includegraphics[width=5.5cm, height=4cm]{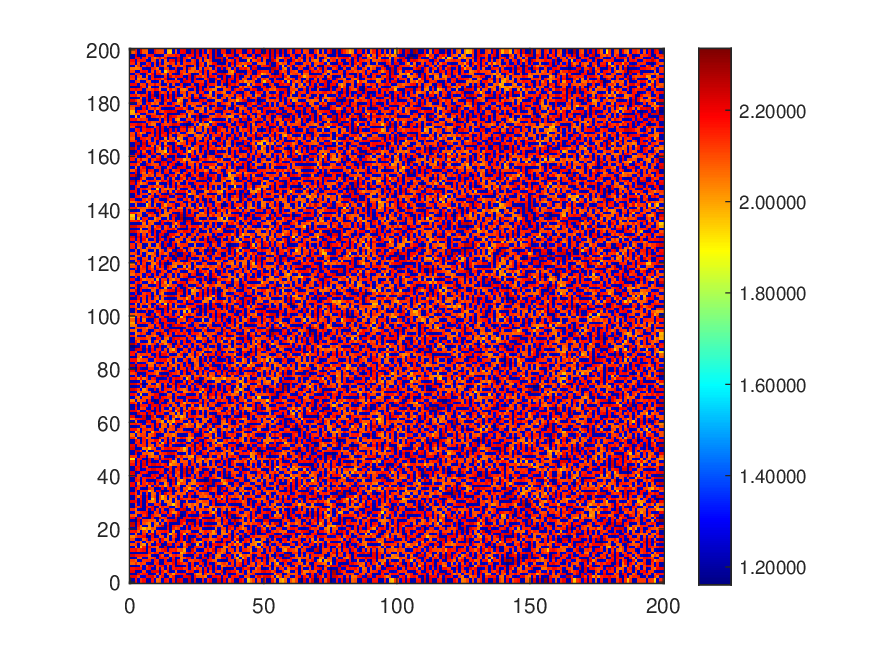}
	\includegraphics[width=5.5cm, height=4cm]{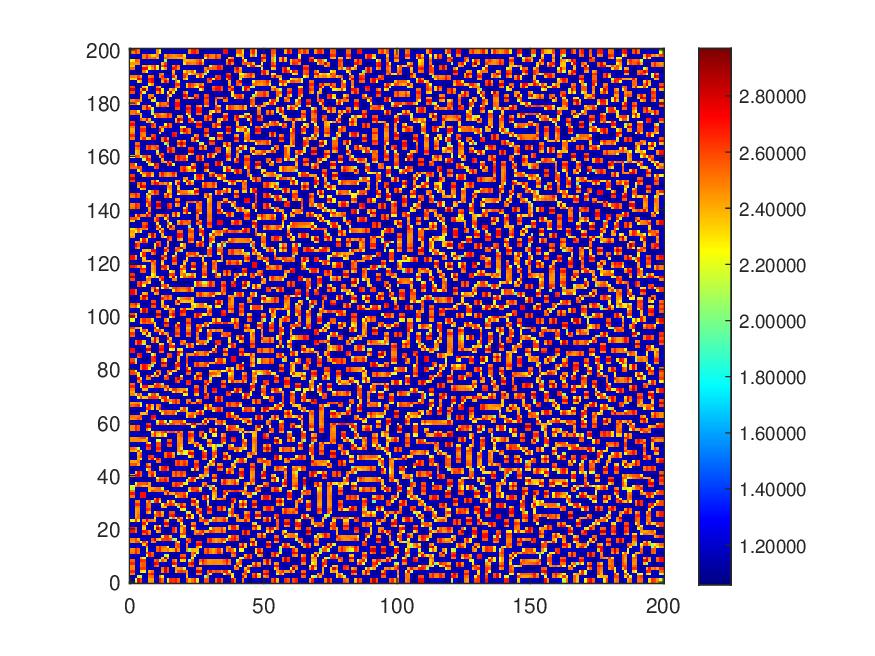}\\
	\subfloat[]{\includegraphics[width=5.5cm, height=4cm]{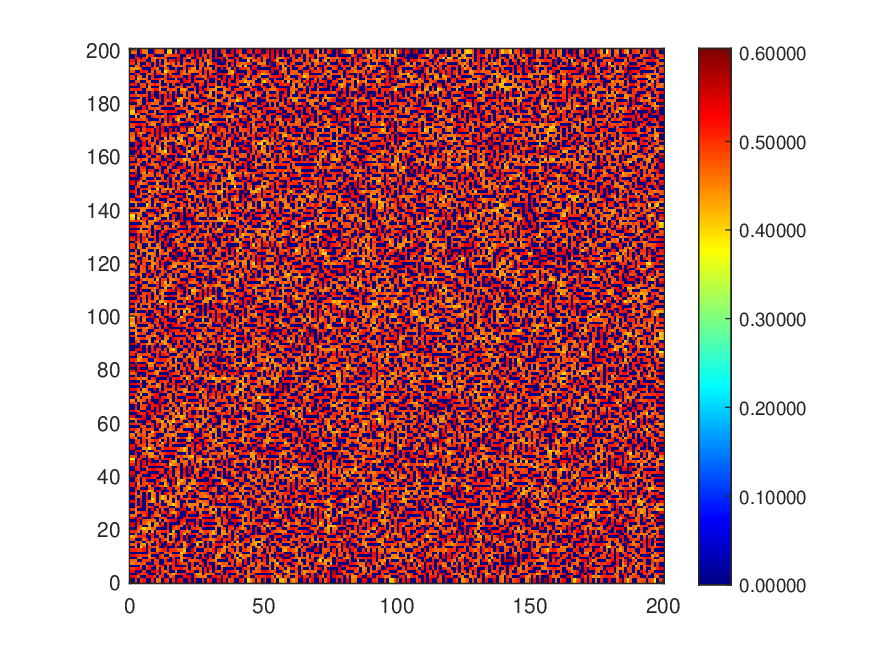}}
	\subfloat[]{\includegraphics[width=5.5cm, height=4cm]{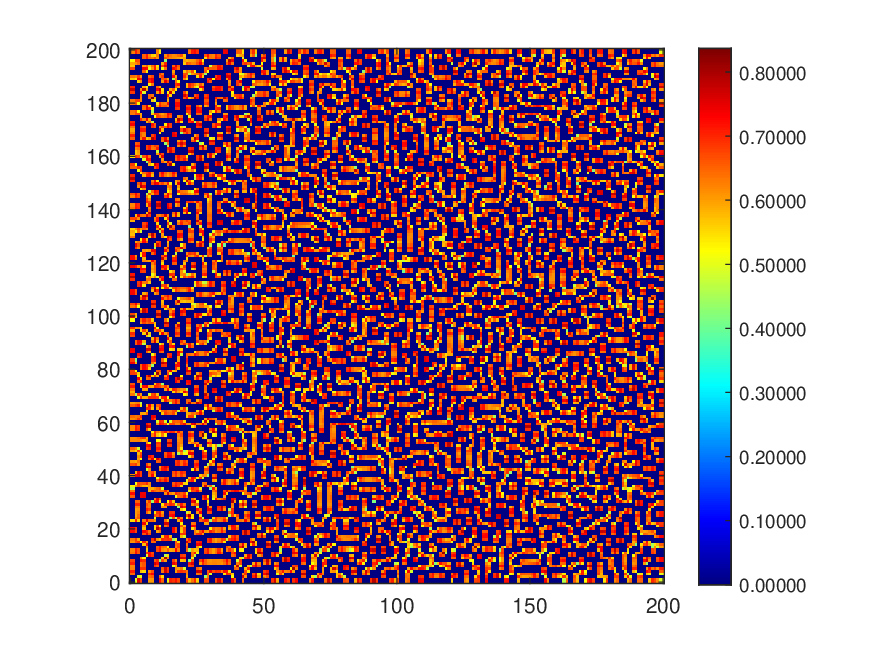}}
	\caption{Spatial distributions of the species Predator $u_1$ (first row) and Prey $u_2$ (second row), where (a, c) $d_{11}=0.1, d_{12}=1, d_{21}=0.1, d_{22}=1,  d_3=3, d_4=3, \sio=2, \sit=3, \lambda_1=2, \lambda_2=1, a_1=a_2=b_1=b_2=5, \eta_1=5, \eta_2=1.5$. \\
		(b, d) $d_{11}=0.1, d_{12}=1, d_{21}=1, d_{22}=2, d_3=3, d_4=2, \sio=2, \sit=3, \lambda_1=2, \lambda_2=1, a_1=a_2=b_1=b_2=5, \eta_1=5, \eta_2=1.5$.}
	\label{fig.5.10}
\end{figure}


\section{Discussions and Conclusion}
This study examined a predator–prey cross-diffusion system coupled with two chemical substances under homogeneous Neumann boundary conditions. Under suitable assumptions on the model parameters, the global existence of classical solutions bounded in the $\lis$ norm for $n \geq 2$ established. Global asymptotic stability of the spatially homogeneous equilibria is derived by constructing an appropriate Lyapunov functional. The analysis leads to the following stability results with biological relevance:
\begin{enumerate}
	\item[(i)] If the predation rate $\eta_2$ satisfies $\eta_2 < \frac{\sit \lambda_1}{\sio}$ and the cross-diffusion coefficients $d_{12}^2$ and $d_{22}^2$ are sufficiently small, the system admits a unique positive equilibrium that is globally asymptotically stable, implying long-term coexistence of prey and predator populations.
	
	\item[(ii)] If $\eta_2 \geq \frac{\sit \lambda_1}{\sio}$ and $d_{12}^2$ is sufficiently small, the semi-trivial equilibrium becomes globally asymptotically stable, indicating eventual extinction of the prey population.
\end{enumerate}
Beyond stability analysis, we investigate diffusion-driven instability and the resulting spatial pattern formation. The results highlight the crucial role of predation intensity in structuring patterns in predator–prey systems mediated by two chemical signals. The patterns shown in Figures \ref{fig.5.5}–\ref{fig.5.10} exhibit a rich spectrum of Turing-type structures arising from the interplay of self-diffusion, cross-diffusion, and nonlinear interactions. As the predation rate $\eta_2$ increases (Figures \ref{fig.5.5}–\ref{fig.5.6}), the system transitions from labyrinthine stripes to mixed stripe–spot patterns and eventually to ordered hexagonal spot arrays, indicating enhanced spatial segregation. Variations in $\eta_1$ and $\eta_2$ (Figure \ref{fig.5.7}) further show that the balance between predator activation and prey response governs pattern geometry, with higher $\eta_1$ promoting localized spot structures. For small values of $\sigma_1, \sigma_2, \lambda_1, \lambda_2,$ and $\eta_2$ (Figure \ref{fig.5.8}), fine-scale spot patterns emerge, reflecting fragmented spatial organization. Moderate increases in $\eta_2$ (Figure \ref{fig.5.9}) yield more regular stripe patterns, suggesting stabilized spatial transport. Under mixed diffusion regimes (Figure \ref{fig.5.10}), patterns become irregular and disordered, indicating proximity to the boundary of Turing instability where no dominant wavelength prevails. Collectively, these results demonstrate how variations in interaction strength and diffusion parameters drive transitions between coherent, localized, and disordered spatial structures.

\section*{Acknowledgments}
GS and JS thank the Anusandhan National Research Foundation (ANRF), formerly Science and Engineering Research Board (SERB), Govt. of India for their support through Core Research Grant (CRG/2023/001483) during this work.

\end{document}